\documentclass[12pt,oneside,english]{amsart}
\usepackage[T1]{fontenc}
\usepackage[latin9]{inputenc}
\usepackage{geometry}
\usepackage{color}
\usepackage{babel}
\usepackage{verbatim}
\usepackage{amsbsy}
\usepackage{amstext}
\usepackage{amsthm}
\usepackage{amssymb}
\usepackage[numbers]{natbib}
\usepackage[unicode=true,pdfusetitle,
 bookmarks=true,bookmarksnumbered=false,bookmarksopen=false,
 breaklinks=false,pdfborder={0 0 1},backref=false,colorlinks=false]
 {hyperref}

\makeatletter
\numberwithin{equation}{section}
\numberwithin{figure}{section}
\theoremstyle{definition}
\newtheorem{defn}{\protect\definitionname}
\theoremstyle{remark}
\newtheorem*{notation*}{\protect\notationname}
\theoremstyle{plain}
\newtheorem{thm}{\protect\theoremname}
\theoremstyle{remark}
\newtheorem*{rem*}{\protect\remarkname}
\theoremstyle{plain}
\newtheorem{assumption}{\protect\assumptionname}
\theoremstyle{plain}
\newtheorem{lem}{\protect\lemmaname}
\theoremstyle{plain}
\newtheorem{prop}{\protect\propositionname}
\theoremstyle{remark}
\newtheorem*{acknowledgement*}{\protect\acknowledgementname}

\usepackage{enumitem}
\setlist{nosep}
\usepackage{graphicx}

\makeatother

\providecommand{\acknowledgementname}{Acknowledgement}
\providecommand{\assumptionname}{Assumption}
\providecommand{\definitionname}{Definition}
\providecommand{\lemmaname}{Lemma}
\providecommand{\notationname}{Notation}
\providecommand{\propositionname}{Proposition}
\providecommand{\remarkname}{Remark}
\providecommand{\theoremname}{Theorem}

\begin{document}
\newcommand{\R}{\mathbb{R}}
\title{Double Obstacle Problems and Fully Nonlinear PDE with Non-strictly
Convex Gradient Constraints}
\author{Mohammad Safdari$\,{}^{1}$}
\begin{abstract}
We prove the optimal $W^{2,\infty}$ regularity for fully nonlinear
elliptic equations with convex gradient constraints. We do not assume
any regularity about the constraints; so the constraints need not
be $C^{1}$ or strictly convex. We also show that the optimal regularity
holds up to the boundary. Our approach is to show that these elliptic
equations with gradient constraints are related to some fully nonlinear
double obstacle problems. Then we prove the optimal $W^{2,\infty}$
regularity for the double obstacle problems. In this process, we also
employ the monotonicity property for the second derivative of obstacles,
which we have obtained in a previous work.\medskip{}

\noindent \textsc{Mathematics Subject Classification.} 35R35, 35J87,
35B65, 49N60.\thanks{$^{1}\;$Department of Mathematical Sciences, Sharif University of
Technology, Tehran, Iran\protect \\
Email address: safdari@sharif.edu}
\end{abstract}

\maketitle

\section{Introduction}

The study of elliptic equations with gradient constraints was initiated
by \citet{MR529814} when he considered the problem 
\[
\max\{Lu-f,\;|Du|-g\}=0,
\]
where $L$ is a linear uniformly elliptic operator of the form 
\[
Lu=-a_{ij}D_{ij}^{2}u+b_{i}D_{i}u+cu.
\]
Equations of this type stem from dynamic programming in a wide class
of singular stochastic control problems. Evans proved $W_{\mathrm{loc}}^{2,p}$
regularity for $u$. He also obtained the optimal $W_{\mathrm{loc}}^{2,\infty}$
regularity under the additional assumption that $a_{ij}$ are constant.
\citet{MR607553} removed this additional assumption and obtained
$W_{\mathrm{loc}}^{2,\infty}$ regularity in general. Later, \citet{MR693645}
allowed the gradient constraint to be more general, and proved global
$W^{2,\infty}$ regularity. We also mention that \citet{soner1989regularity,soner1991free}
considered similar problems with special structure, and proved the
existence of classical solutions.

\citet{yamada1988hamilton} allowed the differential operator to be
more general, and considered the problem 
\[
\max_{1\le k\le N}\{L_{k}u-f_{k},\;|Du|-g\}=0,
\]
where each $L_{k}$ is a linear uniformly elliptic operator. Yamada
proved the existence of a solution in $W_{\mathrm{loc}}^{2,\infty}$.
Recently, there has been new interest in these types of problems.
\citet{hynd2013analysis} considered problems with more general gradient
constraints of the form 
\[
\max\{Lu-f,\;\tilde{H}(Du)\}=0,
\]
where $\tilde{H}$ is a convex function. He proved $W_{\mathrm{loc}}^{2,\infty}$
regularity when $\tilde{H}$ is strictly convex. Finally, \citet{Hynd}
studied fully nonlinear elliptic equations with strictly convex gradient
constraints of the form 
\[
\max\{F(x,D^{2}u)-f,\;\tilde{H}(Du)\}=0.
\]
Here $F(x,D^{2}u)$ is a fully nonlinear elliptic operator. They obtained
$W_{\mathrm{loc}}^{2,p}\cap W^{1,\infty}$ regularity in general,
and $W_{\mathrm{loc}}^{2,\infty}$ regularity when $F$ does not depend
on $x$. Let us also mention that \citet{Hynd-2012,Hynd2017} considered
eigenvalue problems for equations with gradient constraints too.

Closely related to the above problems are variational problems with
gradient constraints. An important example among them is the well-known
elastic-plastic torsion problem, which is the problem of minimizing
the functional 
\[
\int_{U}\frac{1}{2}|Dv|^{2}-v\,dx
\]
over the set 
\[
W_{B_{1}}:=\{v\in W_{0}^{1,2}(U):|Dv|\le1\textrm{ a.e.}\}.
\]
Here $U$ is a bounded open set in $\mathbb{R}^{n}$. This problem
is equivalent to finding $u\in W_{B_{1}}$ that satisfies the variational
inequality 
\[
\int_{U}Du\cdot D(v-u)-(v-u)\,dx\ge0\qquad\textrm{ for every }v\in W_{B_{1}}.
\]
An interesting property of variational problems with gradient constraints
is that under mild conditions they are equivalent to double obstacle
problems. For example, $u$, the minimizer of 
\begin{equation}
J[v]:=\int_{U}G(Dv)+g(v)\,dx\label{eq: fnctnl J}
\end{equation}
over $W_{B_{1}}$, also satisfies $-d\le u\le d$, and 
\[
\begin{cases}
-D_{i}(D_{i}G(Du))+g'(u)=0 & \textrm{ in }\{-d<u<d\},\\
-D_{i}(D_{i}G(Du))+g'(u)\le0 & \textrm{ a.e. on }\{u=d\},\\
-D_{i}(D_{i}G(Du))+g'(u)\ge0 & \textrm{ a.e. on }\{u=-d\},
\end{cases}
\]
where $d$ is the Euclidean distance to $\partial U$; see for example
\citep{MR1,SAFDARI202176}.

\citet{MR0239302} proved the $W^{2,p}$ regularity for the elastic-plastic
torsion problem. \citet{MR513957} obtained its optimal $W_{\mathrm{loc}}^{2,\infty}$
regularity. \citet{MR0385296} proved $W^{2,p}$ regularity for the
solution of a quasilinear variational inequality subject to the same
constraint as in the elastic-plastic torsion problem. \citet{MR697646}
proved $W^{2,p}$ regularity for the solution of a linear variational
inequality subject to a $C^{2}$ strictly convex gradient constraint.
\citet{MR1310935,MR1315349} proved $C^{1,\alpha}$ regularity for
the solution to a quasilinear variational inequality subject to a
$C^{2}$ strictly convex gradient constraint, and allowed the operator
to be degenerate of the $p$-Laplacian type. The paper \citep{MR1315349}
is also the only work here that is concerned with functionals with
non-quadratic $p$-growth.

Variational problems with gradient constraints have also seen new
developments in recent years. By using infinite dimensional duality,
\citet{giuffre2015lagrange} studied the Lagrange multipliers of quasilinear
variational inequalities subject to the same constraint as in the
elastic-plastic torsion problem. \citet{MR2605868} investigated the
minimizers of some functionals subject to gradient constraints, arising
in the study of random surfaces. In their work, the functional is
allowed to have certain kinds of singularities, and there is no particular
assumption about its growth condition. Also, the constraints are given
by convex polygons; so they are not strictly convex. They showed that
in two dimensions, the minimizer is $C^{1}$ away from the obstacles.
\citet{choe2016elliptic} generalized the regularity results of \citep{MR1310935,MR1315349}
by allowing more general constraints.

In {[}\citealp{Safdari20151}\nocite{safdari2017shape,MR1}\textendash \citealp{MR1},
\citealp{SAFDARI202176}{]} we have studied the regularity and the
free boundary of several classes of variational problems with gradient
constraints. Our goal was to understand the behavior of these problems
when the constraint is not strictly convex; and we have been able
to obtain the optimal $W^{2,\infty}$ regularity for them. This has
been partly motivated by the above-mentioned problem about random
surfaces. There are also similar interests in elliptic equations with
gradient constraints which are not strictly convex. These problems
emerge in the study of some singular stochastic control problems appearing
in financial models with transaction costs; see for example \citep{barles1998option,possamai2015homogenization}.

In this paper, we obtain a link between double obstacle problems and
elliptic equations with gradient constraints. This link has been well
known in the case where the double obstacle problem reduces to an
obstacle problem. However, we will show that there is still a connection
between the two problems in the general case. This connection allows
us to obtain the optimal $W^{2,\infty}$ regularity for fully nonlinear
elliptic equations which do not depend explicitly on $x$, and are
subject to non-strictly convex gradient constraints. It also paves
the way for studying more general elliptic equations with such constraints.
In this approach, we will also study fully nonlinear double obstacle
problems with singular obstacles, and we will obtain the optimal $W^{2,\infty}$
regularity for them. These types of singular obstacles have not been
studied before, to the best of author's knowledge.%
\begin{comment}
more explanation ??
\end{comment}
{} However, see \citep{MR2989443,lee2019regularity} for some recent
works on double obstacle problems.

Let us introduce the problem in more detail. First let us recall some
concepts from convex analysis. Let $K$ be a compact convex subset
of $\mathbb{R}^{n}$ whose interior contains the origin.
\begin{defn}
The \textbf{gauge} function of $K$ is the function 
\begin{equation}
\gamma_{K}(x):=\inf\{\lambda>0:x\in\lambda K\}.\label{eq: gaug}
\end{equation}
And the \textbf{polar} of $K$ is the set 
\begin{equation}
K^{\circ}:=\{x:\langle x,y\rangle\leq1\,\textrm{ for all }y\in K\},\label{eq: K0}
\end{equation}
where $\langle\,,\rangle$ is the standard inner product on $\mathbb{R}^{n}$.
\end{defn}
The gauge function $\gamma_{K}$ is convex, subadditive, and positively
1-homogeneous; so it looks like a norm on $\mathbb{R}^{n}$, except
that $\gamma_{K}(-x)$ is not necessarily the same as $\gamma_{K}(x)$.
The polar set $K^{\circ}$ is also a compact convex set containing
the origin as an interior point. Also note that since $K$ is closed
we have $K=\{\gamma_{K}\le1\}$, and since $K$ has nonempty interior
we have $\partial K=\{\gamma_{K}=1\}$. (For more details, and the
proofs of these facts, see \citep{MR3155183}.)

Let $U\subset\mathbb{R}^{n}$ be a bounded open set with Lipschitz
boundary. Let 
\begin{equation}
W_{K^{\circ},\varphi}=W_{K^{\circ},\varphi}(U):=\{v\in W^{1,2}(U):Dv\in K^{\circ}\textrm{ a.e., }v=\varphi\textrm{ on }\partial U\}.\label{eq: W_K}
\end{equation}
Here $\varphi:\mathbb{R}^{n}\to\mathbb{R}$ is a continuous function,
and the equality of $v,\varphi$ on $\partial U$ is in the sense
of trace. In order to ensure that $W_{K^{\circ},\varphi}$ is nonempty
we assume that 
\begin{equation}
-\gamma_{K}(y-x)\le\varphi(x)-\varphi(y)\le\gamma_{K}(x-y),\label{eq: phi Lip}
\end{equation}
for all $x,y\in\R^{n}$. Then by Lemma 2.1 of \citep{MR1797872} this
property implies that $\varphi$ is Lipschitz and $D\varphi\in K^{\circ}$
a.e.; so $\varphi\in W_{K^{\circ},\varphi}$.

Also let 
\begin{equation}
W_{\bar{\rho},\rho}=W_{\bar{\rho},\rho}(U):=\{v\in W^{1,2}(U):-\bar{\rho}\le v\leq\rho\textrm{ a.e., }v=\varphi\textrm{ on }\partial U\},\label{eq: W_rho}
\end{equation}
where the obstacles are 
\begin{align}
 & \rho(x)=\rho_{K,\varphi}(x;U):=\underset{y\in\partial U}{\min}[\gamma_{K}(x-y)+\varphi(y)],\nonumber \\
 & \bar{\rho}(x)=\bar{\rho}_{K,\varphi}(x;U):=\underset{y\in\partial U}{\min}[\gamma_{K}(y-x)-\varphi(y)].\label{eq: rho}
\end{align}
It is well known (see {[}\citealp{MR667669}, Section 5.3{]}) that
$\rho$ is the unique viscosity solution of the Hamilton-Jacobi equation
\begin{equation}
\begin{cases}
\gamma_{K^{\circ}}(Dv)=1 & \textrm{in }U,\\
v=\varphi & \textrm{on }\partial U.
\end{cases}\label{eq: H-J eq}
\end{equation}
Now, note that $-K$ is also a compact convex set whose interior contains
the origin. We also have $\bar{\rho}_{K,\varphi}=\rho_{-K,-\varphi}$,
since $\gamma_{-K}(\cdot)=\gamma_{K}(-\,\cdot)$. Thus we have a similar
characterization for $\bar{\rho}$ too.
\begin{notation*}
To simplify the notation, we will use the following conventions 
\[
\gamma:=\gamma_{K},\qquad\gamma^{\circ}:=\gamma_{K^{\circ}},\qquad\bar{\gamma}:=\gamma_{-K}.
\]
Thus in particular we have $\bar{\gamma}(x)=\gamma(-x)$.
\end{notation*}
In \citep{SAFDARI202176} we have shown that $-\bar{\rho}\le\rho$,
and
\begin{equation}
-\gamma(x-y)\le\rho(y)-\rho(x)\le\gamma(y-x),\label{eq: rho Lip}
\end{equation}
for all $x,y\in\R^{n}$. The above inequality also holds if we replace
$\rho,\gamma$ with $\bar{\rho},\bar{\gamma}$. Thus in particular,
$\rho,\bar{\rho}$ are Lipschitz continuous. We have also shown that
$-\bar{\rho},\rho\in W_{K^{\circ},\varphi}(U)$, and 
\[
W_{K^{\circ},\varphi}(U)\subset W_{\bar{\rho},\rho}(U).
\]
In addition, we showed that $u$, the minimizer of the functional
$J$ over $W_{K^{\circ},\varphi}$, is also the minimizer of $J$
over $W_{\bar{\rho},\rho}$. We also proved that under appropriate
assumptions $u$ belongs to $W^{2,\infty}$, without requiring any
smoothness or strict convexity about the gradient constraint $K^{\circ}$.

Motivated by the double obstacle problems arising from variational
problems, we are going to study the fully nonlinear double obstacle
problem 
\begin{equation}
\begin{cases}
F[u]=0 & \textrm{ a.e. in }\{-\bar{\rho}<u<\rho\},\\
F[u]\le0 & \textrm{ a.e. on }\{u=\rho\},\\
F[u]\ge0 & \textrm{ a.e. on }\{u=-\bar{\rho}\},
\end{cases}\label{eq: dbl obstcl}
\end{equation}
and employ it to better understand elliptic equations with gradient
constraints. Here we have used the convention 
\[
F[u]:=F(x,u,Du,D^{2}u).
\]

The assumptions about $F$ are stated in the next section. We stated
the assumptions in the general setting in which $F$ can also depend
on $x$. Because we will consider general double obstacle problems
with more regular obstacles in the appendix. And we also wanted to
determine which parts of our proof still work when $F$ depends on
$x$, to facilitate future works on such equations.%
\begin{comment}
specify exactly where it fails ? (in Theorem 2 and early in proof
of Theorem 3, using propositions and lemmas of that section. Essentially
when we use the full force of the maximum principle)
\end{comment}
{} However, our main results in this paper are about the case where
$F$ does not depend on $x$, of which a prototypical example is 
\[
\max\{F(D^{2}u),\;\tilde{H}(Du)\}=0,
\]
where the set $\{\tilde{H}(\cdot)\le0\}$ is a convex polyhedron.
Here the assumptions on $F$ can simply be stated as: $F$ is a $C^{2}$
convex function satisfying $F(0)=0$, and it is uniformly elliptic
(see Assumption \ref{assu: F}). This example resembles the equations
emerging from some singular stochastic control problems appearing
in financial models with transaction costs (see \citep{barles1998option,possamai2015homogenization}).

Now let us state our main results.
\begin{thm}
\label{thm: Reg PDE grad}Suppose $F$ does not depend on $x$, and
satisfies Assumptions \ref{assu: F},\ref{assu: F C2},\ref{assu: F Lip}.
Also suppose $\partial U$ is $C^{2,\alpha}$ for some $\alpha>0$,
and $\varphi$ is $C^{2,\alpha}$ with $\gamma^{\circ}(D\varphi)<1$.
In addition, suppose there is $v\in C^{0}(\overline{U})\cap W_{\mathrm{loc}}^{2,n}(U)\cap W_{\bar{\rho},\rho}(U)$
that satisfies $F[v]\le0$ a.e. Let $H=\gamma^{\circ}-1$. Then there
is $u\in W^{2,\infty}(U)$ that satisfies the elliptic equation with
gradient constraint 
\begin{equation}
\begin{cases}
\max\{F(u,Du,D^{2}u),\;H(Du)\}=0 & \textrm{ a.e. in }U,\\
u=\varphi & \textrm{ on }\partial U.
\end{cases}\label{eq: PDE grad constr}
\end{equation}
\end{thm}
\begin{rem*}
Note that if the above equation with gradient constraint has a solution
then we must have a subsolution ($F\le0$) inside $W_{K^{\circ},\varphi}\subset W_{\bar{\rho},\rho}$.
Thus the existence of $v$ is a natural requirement. In particular,
note that this requirement is weaker than a corresponding condition
in \citep{Hynd}, which requires the existence of a ``strict'' subsolution
of (\ref{eq: PDE grad constr}) in $C^{2}$.%
\begin{comment}
\textendash{} comment on the variational case and the less regularity
which is needed there in the two results below ?

\textendash{} UNIQUENESS for equation with gradient constraint ???
\end{comment}
\end{rem*}
\begin{rem*}
Note that we are not assuming any regularity about $\partial K$ or
$\partial K^{\circ}$. In particular, $\gamma^{\circ}$, which defines
the gradient constraint, need not be $C^{1}$ or strictly convex.
Furthermore, note that any convex gradient constraint of the general
form $\tilde{H}(Du)\le0$ for which the set $\{\tilde{H}(\cdot)\le0\}$
is bounded, and contains a neighborhood of the origin (which is a
natural requirement in these problems), can be written in the form
$\gamma^{\circ}-1$ with respect to the convex set $K=\{\tilde{H}(\cdot)\le0\}^{\circ}$.
(Note that $\{\tilde{H}(\cdot)\le0\}=K^{\circ}$, because the double
polar of such convex sets are themselves, as shown in Section 1.6
of \citep{MR3155183}.)
\end{rem*}
In contrast to the regularity result of \citep{Hynd}, the main difference
of our result is that we do not require the gradient constraint to
be strictly convex. However, we do not allow $F$ to depend on $x$
(although we allow dependence on $u,Du$). This is mainly because
we need the full power of the maximum principle for $Du$ on several
occasions, at which mere estimates of $|Du|$ are not sufficient.
Another difference is that here we obtain optimal regularity up to
the boundary in addition to local regularity. We should mention that
our technique, even in case of local regularity, is inherently global.
Because we use the behavior of the obstacles at $\partial U$ in a
crucial way. In particular we employ Lemma \ref{lem: D2 rho decreas},
which is a monotonicity property for $D^{2}\rho,D^{2}\bar{\rho}$.

We split the proof of Theorem \ref{thm: Reg PDE grad} into two parts.
In Theorem \ref{thm: Reg dbl obstcl} we show that there is $u\in W^{2,\infty}(U)$
that satisfies the double obstacle problem (\ref{eq: dbl obstcl}).
And in Theorem \ref{thm: dbl obstcl =00003D PDE grad} we prove that
$u$ must also satisfy the elliptic equation with gradient constraint
(\ref{eq: PDE grad constr}).
\begin{thm}
\label{thm: dbl obstcl =00003D PDE grad}Suppose $F$ does not depend
on $x$, and satisfies Assumptions \ref{assu: F},\ref{assu: F C2}.
Also suppose $\partial U$ is $C^{1}$, and there is $v\in C^{0}(\overline{U})\cap W_{\mathrm{loc}}^{2,n}(U)\cap W_{\bar{\rho},\rho}(U)$
that satisfies $F[v]\le0$ a.e. Let $u\in C^{1}(\overline{U})\cap W_{\mathrm{loc}}^{2,n}(U)\cap W_{\bar{\rho},\rho}(U)$
be a solution of the double obstacle problem (\ref{eq: dbl obstcl}).
Then $u$ also satisfies the elliptic equation with gradient constraint
(\ref{eq: PDE grad constr}).
\end{thm}
\begin{thm}
\label{thm: Reg dbl obstcl}Suppose $F$ does not depend on $x$,
and satisfies Assumptions \ref{assu: F},\ref{assu: F C2},\ref{assu: F Lip}.
Also suppose $\partial U$ is $C^{2,\alpha}$ for some $\alpha>0$.
In addition, suppose that $\varphi$ is $C^{2,\alpha}$, and $\gamma^{\circ}(D\varphi)<1$.
Then there is $u\in W^{2,\infty}(U)\cap W_{\bar{\rho},\rho}(U)$ that
satisfies the double obstacle problem (\ref{eq: dbl obstcl}).
\end{thm}
\begin{rem*}
Note that we are not assuming any regularity about $\partial K$ or
$\partial K^{\circ}$; so the obstacles can be highly irregular.
\end{rem*}
\begin{rem*}
In Theorem \ref{thm: Reg dbl obstcl}, and hence in Theorem \ref{thm: Reg PDE grad},
we can assume $\varphi$ is $C^{1,1}$. In this case, in the proof
of Theorem \ref{thm: Reg dbl obstcl} we have to work with smooth
approximations of $\varphi$ instead of $\varphi$. The required modifications
can be done easily; however, to keep the presentation of the proof
more clear we preferred to not follow this route.%
\begin{comment}
\textendash{} ???

\textendash{} state this for the variational case in the other paper
- Does it hold for the 2nd theorem about convex domains too ??

\textendash{} say that $\gamma^{\circ}(D\varphi)<1$ ? Is this really
needed ? (apparently not)

\textendash{} say that $\partial U$ still needs to be $C^{2,\alpha}$,
since we need to construct $u_{k}$ by solving fully nonlinear equations
? ??
\end{comment}
\end{rem*}
\begin{rem*}
Theorem \ref{thm: Reg PDE grad} and Theorem \ref{thm: Reg dbl obstcl}
hold under the more relaxed condition $\gamma^{\circ}(D\varphi)\le1$;
however, then we also need to impose some technical restrictions at
those points of $\partial U$ at which $\gamma^{\circ}(D\varphi)=1$.
The details of this more general case are presented before the proof
of Theorem \ref{thm: Reg dbl obstcl} on page \pageref{enu: cond *}.
\end{rem*}
The paper is organized as follows. In Section \ref{sec: Prelim} we
first state our assumptions, and prove Theorem \ref{thm: dbl obstcl =00003D PDE grad}.
Here we use Lemma \ref{lem: u<v} in which we prove that a solution
to the double obstacle problem (\ref{eq: dbl obstcl}) must be larger
than the solution of the elliptic equation with gradient constraint
(\ref{eq: PDE grad constr}). Then we review some well-known facts
about the regularity of $K$, and its relation to the regularity of
$K^{\circ},\gamma,\gamma^{\circ}$. After that we consider the function
$\rho$ more carefully. We will review the formulas for the derivatives
of $\rho$ that we have obtained in \citep{SAFDARI202176}, especially
the novel explicit formula (\ref{eq: D2 rho (x)}) for $D^{2}\rho$.
To the best of author's knowledge, formulas of this kind have not
appeared in the literature before, except for the simple case where
$\rho$ is the Euclidean distance to the boundary. (Although, some
special two dimensional cases also appeared in our earlier works \citep{Safdari20151,Paper-4}.)

One of the main applications of the formula (\ref{eq: D2 rho (x)})
for $D^{2}\rho$ is in the relation (\ref{eq: ridge =00003D Q=00003D0})
for characterizing the set of singularities of $\rho$. Another important
application is in Lemma \ref{lem: D2 rho decreas}, which implies
that $D^{2}\rho$ attains its maximum on $\partial U$. This interesting
property is actually a consequence of a more general property of the
solutions to Hamilton-Jacobi equations (remember that $\rho$ is the
viscosity solution of the Hamilton-Jacobi equation (\ref{eq: H-J eq})).
This little-known monotonicity property is investigated in \citep{SAFDARI202176};
but we included a brief account at the end of Section \ref{sec: Prelim}
for reader's convenience.

In Section \ref{sec: Proof Thm 3} we prove the regularity result
for double obstacle problem (\ref{eq: dbl obstcl}), aka Theorem \ref{thm: Reg dbl obstcl}.
The idea of the proof of this theorem is to approximate $K^{\circ}$
with smoother convex sets. Then, as it is common in the study of the
regularity of PDEs, we have to find uniform bounds for the various
norms of the approximations $u_{k}$ to $u$. Here, among other estimations,
we will use the fact that the second derivative of the approximations
$\rho_{k}$ to $\rho$ attain their maximums on $\partial U$. We
will also use our detailed knowledge of the set of singularities of
$\rho_{k}$ to show that $u_{k}$ does not touch $\rho_{k}$ at its
singularities (see Proposition \ref{prop: ridge is elastic}). Let
us finally mention that in order to get the optimal $W^{2,\infty}$
regularity we used the result of \citet{MR3198649}, and its generalizations
by \citet{Indrei-Minne,indrei2016nontransversal}. A more detailed
sketch of proof for Theorem \ref{thm: Reg dbl obstcl} is given at
the beginning of its proof (before its Part I) on page \pageref{proof Thm 3}.

At the end, in Appendix \ref{sec: Appdx}, we obtain a standard regularity
result for double obstacle problems, which we have used in the article.
Here the obstacles are more regular. We also allowed $F$ to explicitly
depend on $x$. The penalization method employed in the appendix is
classical, but to the best of author\textquoteright s knowledge the
results have not appeared elsewhere. Nevertheless, we include the
proofs here for completeness. \vspace{1.5cm}

\section{\label{sec: Prelim}Assumptions and Preliminaries}

First let us introduce some more notation. We will use $d(x):=\min_{y\in\partial U}|x-y|$
to denote the Euclidean distance to $\partial U$. In addition, for
two points $x,y\in\R^{n}$, $[x,y],\,]x,y[,\,[x,y[,\,]x,y]$ will
denote the closed, open, and half-open line segments with endpoints
$x,y$, respectively. Also, we will use the convention of summing
over repeated indices.

Remember that a \textit{strong solution} of a second order equation
is a $W^{2,p}$ function that satisfies the equation a.e. We will
also use the notion of viscosity solution, so we are going to review
its definition.
\begin{defn}
A continuous function $u$ is a \textit{viscosity solution} of $F[u]=0$
if whenever $\phi$ is a $C^{2}$ function and $u-\phi$ has a local
maximum at an interior point $x_{0}$ we have 
\[
F(x_{0},u(x_{0}),D\phi(x_{0}),D^{2}\phi(x_{0}))\le0,
\]
and whenever $\psi$ is a $C^{2}$ function and $u-\psi$ has a local
minimum at an interior point $x_{0}$ we have 
\[
F(x_{0},u(x_{0}),D\psi(x_{0}),D^{2}\psi(x_{0}))\ge0.
\]
\end{defn}
Now let us state our main assumptions about $F$. These are standard
assumptions on the convexity and monotonicity of $F$, and the growth
rates of $F$ and its derivatives, needed for the study of regularity
of fully nonlinear equations. In the following, $\mathcal{S}^{n\times n}$
denotes the space of symmetric $n\times n$ real matrices.
\begin{assumption}
\label{assu: F}The function $F(x,z,p,M):\overline{U}\times\mathbb{R}\times\mathbb{R}^{n}\times\mathcal{S}^{n\times n}\to\mathbb{R}$
is a $C^{1}$ function that satisfies\medskip{}
\begin{enumerate}
\item[\upshape{(a)}] $F$ is uniformly elliptic, i.e. there are constants $\Lambda\ge\lambda>0$
such that 
\[
-\Lambda\,\mathrm{tr}(N)\le F(x,z,p,M+N)-F(x,z,p,M)\le-\lambda\,\mathrm{tr}(N),
\]
for all $x\in\overline{U}$, $z\in\mathbb{R}$, $p\in\mathbb{R}^{n}$,
and $M,N\in\mathcal{S}^{n\times n}$ with $N\ge0$.
\item[\upshape{(b)}] For every $C>0$ there is $c_{1}=c_{1}(C)>0$ such that 
\begin{equation}
\begin{cases}
|F(x,z,p,0)|\le c_{1}(1+|p|^{2}),\\
|F_{x}|,\;|F_{z}|\le c_{1}(1+|p|^{2}+|M|),\\
|F_{p}|\le c_{1}(1+|p|+|M|),
\end{cases}\label{eq: Bnds DF}
\end{equation}
for all $x\in\overline{U}$, $|z|\le C$, $p\in\mathbb{R}^{n}$, and
$M\in\mathcal{S}^{n\times n}$.
\item[\upshape{(c)}] $F$ is an increasing function of $z$ for each fixed $(x,p,M)$,
i.e. $F_{z}\ge0$.
\item[\upshape{(d)}] $F$ is a convex function of $M$.
\item[\upshape{(e)}] We have%
\begin{comment}
$F(x,0,0,0)=0$, and
\end{comment}
\[
F(x,z,p,0)\,\mathrm{sign}(z)\ge-c_{3}(1+|p|)
\]
for all $x\in\overline{U}$, $z\in\mathbb{R}$, $p\in\mathbb{R}^{n}$,
and some constant $c_{3}>0$.
\end{enumerate}
\end{assumption}
\begin{assumption}
\label{assu: F C2}Suppose that $F$ is $C^{2}$, and for every $C>0$
there is $c_{2}=c_{2}(C)>0$ such that 
\begin{equation}
\begin{cases}
|F_{Mx}|,|F_{Mz}|,|F_{Mp}|\le c_{2},\\
|F_{pp}|,|F_{pz}|,|F_{px}|,|F_{zz}|,|F_{zx}|,|F_{xx}|\le c_{2}(1+|M|),
\end{cases}\label{eq: Bnds D2F}
\end{equation}
for all $x\in\overline{U}$, $|z|,|p|\le C$, and $M\in\mathcal{S}^{n\times n}$.
\end{assumption}
\begin{lem}
\label{lem: u<v}Suppose $F$ satisfies Assumption \ref{assu: F},
and $u\in C^{0}(\overline{U})\cap W_{\mathrm{loc}}^{2,n}(U)\cap W_{\bar{\rho},\rho}(U)$
is a solution of the double obstacle problem (\ref{eq: dbl obstcl}).
Also suppose that $v\in C^{0}(\overline{U})\cap W_{\mathrm{loc}}^{2,n}(U)\cap W_{\bar{\rho},\rho}(U)$
satisfies $F[v]\le0$ a.e. Then we have 
\[
v\le u.
\]
As a result we get 
\[
\begin{cases}
F[u]=0 & \textrm{ a.e. in }\{-\bar{\rho}\le u<\rho\},\\
F[u]\le0 & \textrm{ a.e. on }\{u=\rho\}.
\end{cases}
\]
\end{lem}
\begin{rem*}
In fact, this lemma is still true if we replace $\rho,-\bar{\rho}$
by any other upper and lower obstacles which agree on $\partial U$.
We can also replace the $0$ on the right hand sides by some measurable
function $f$. In addition, parts (d),(e) of Assumption \ref{assu: F}
are not needed here.
\end{rem*}
\begin{proof}
Let $w:=v-u$. Then on the open set $V:=\{u<\rho\}$ we have 
\[
0\ge F[v]-F[u]=\int_{0}^{1}\frac{d}{dt}\big(F[u+tw]\big)dt=-a_{ij}D_{ij}^{2}w+b_{i}D_{i}w+cw,
\]
where $a_{ij}:=-\int_{0}^{1}F_{M_{ij}}[u+tw]dt$, $b_{i}:=\int_{0}^{1}F_{p_{i}}[u+tw]dt$,
and $c:=\int_{0}^{1}F_{z}[u+tw]dt$. Note that on $\partial V$ we
have $u=\rho\ge v$, so $w\le0$. Hence by Aleksandrov-Bakelman-Pucci
maximum principle (Theorem 9.1 of \citep{MR1814364}) we get 
\[
\sup_{V}w\le\sup_{\partial V}w^{+}=0.
\]
Thus $v-u\le0$ as desired. Finally note that when $u=-\bar{\rho}$
we have $u\ge v\ge-\bar{\rho}$, thus $v=-\bar{\rho}$ too. Therefore
we have $F[u]=F[v]\le0$ a.e. on $\{u=-\bar{\rho}\}$. Hence we must
have $F[u]=0$ a.e. on $\{u=-\bar{\rho}\}$.
\end{proof}
\begin{proof}[\textbf{Proof of Theorem \ref{thm: dbl obstcl =00003D PDE grad}}]
\label{proof Thm 2} By Lemma \ref{lem: u<v} we know that 
\[
\begin{cases}
F[u]=0 & \textrm{ a.e. in }\{-\bar{\rho}\le u<\rho\},\\
F[u]\le0 & \textrm{ a.e. on }\{u=\rho\}.
\end{cases}
\]
Hence we have $F[u]\le0$. Also, on $\{u=\rho\}$ we have $Du=D\rho$,
since $u-\rho$ attains its maximum there. But we know that $\gamma^{\circ}(D\rho)=1$
a.e. (see (\ref{eq: H-J eq})). Therefore when $\gamma^{\circ}(Du)<1$
we must have $F[u]=0$ a.e. Thus we only need to show that $\gamma^{\circ}(Du)\le1$
a.e. in $V=\{u<\rho\}$.

Now note that for any ball $B\subset V$ there is a $C^{2,\alpha}(B)\cap C^{0}(\overline{B})$
solution of 
\[
\begin{cases}
F[w]=0 & \textrm{ in }B,\\
w=u & \textrm{ on }\partial B,
\end{cases}
\]
as shown in \citep{trudinger1984boundary}. However, similarly to
the proof of Lemma \ref{lem: u<v} we can show that due to the Aleksandrov-Bakelman-Pucci
maximum principle we have $w=u$ on $B$. Therefore $u$ is $C^{2,\alpha}$
inside $V$. Thus by Lemma 17.16 of \citep{MR1814364} we have $u\in C^{3,\alpha}(V)$.
Let $\xi\in\mathbb{R}^{n}$ be a vector with $\gamma(\xi)=1$, and
differentiate the equation $F[u]=0$ to obtain 
\[
F_{z}[u]D_{\xi}u+F_{p_{i}}[u]D_{i}D_{\xi}u+F_{M_{ij}}[u]D_{ij}^{2}D_{\xi}u=0.
\]
Now on $\partial V\cap U$ we have $Du=D\rho$, so $D_{\xi}u=D_{\xi}\rho\le1$
due to (\ref{eq: gen Cauchy-Schwartz}). On $\partial V\cap\partial U$
we need to consider two cases. Since $-\bar{\rho}\le u\le\rho$ in
$U$ and $u=\varphi=\rho$ on $\partial U$, when $\xi$ is not tangent
to $\partial U$ we have 
\[
-\bar{\rho}(y+t\xi)-(-\bar{\rho}(y))\le u(y+t\xi)-u(y)\le\rho(y+t\xi)-\rho(y),
\]
where $y\in\partial U$, and the sign of $t$ is chosen so that $y+t\xi\in U$.
Hence by dividing by $t$ and taking $t\to0$ we either obtain $D_{\xi}u\le D_{\xi}\rho\le1$,
or $D_{\xi}u\le-D_{\xi}\bar{\rho}\le1$. On the other hand, when $\xi$
is tangent to $\partial U$ we have $D_{\xi}u=D_{\xi}\rho\le1$, since
$u=\varphi=\rho$ on $\partial U$. Thus we have shown that $D_{\xi}u\le1$
on $\partial V$. Hence by the maximum principle we have $D_{\xi}u\le1$
in $V$. Therefore by (\ref{eq: gen Cauchy-Schwartz 2}) we get $\gamma^{\circ}(Du)\le1$
in $V$, as desired.
\end{proof}
We will later need the following additional assumption about $F$
to make sure that $W^{2,p}$ estimates hold for the solutions of the
equation $F[u]=0$ (see for example \citep{winter2009w2}). However,
any other assumption that gives us the $W^{2,p}$ estimates can also
be used instead.
\begin{assumption}
\label{assu: F Lip}For every $x\in\overline{U}$ we have $F(x,0,0,0)=0$.
Also, $F$ is uniformly elliptic and Lipschitz, i.e. there are constants
$c_{4},c_{5}>0$ such that 
\begin{align}
\mathcal{P}^{-}(M-N)-c_{4}|p-q|-c_{5}|z-w| & \le F(x,z,p,M)-F(x,w,q,N)\nonumber \\
 & \qquad\qquad\le\mathcal{P}^{+}(M-N)+c_{4}|p-q|+c_{5}|z-w|\label{eq: F Lip}
\end{align}
for all $z,w\in\mathbb{R}$, $p,q\in\mathbb{R}^{n}$, and $M,N\in\mathcal{S}^{n\times n}$.
Here $\mathcal{P}^{\pm}$ are the Pucci operators 
\[
\mathcal{P}^{-}(M):=\inf_{\lambda I\le A\le\Lambda I}\mathrm{tr}(AM),\qquad\mathcal{P}^{+}(M):=\sup_{\lambda I\le A\le\Lambda I}\mathrm{tr}(AM),
\]
and $\lambda,\Lambda>0$ are the same as in Assumption \ref{assu: F}.
\end{assumption}
Next let us introduce the following terminology for the solutions
of the double obstacle problem (\ref{eq: dbl obstcl}). (The notation
is motivated by the physical properties of the elastic-plastic torsion
problem, in which $E$ stands for the \textit{elastic} region, and
$P$ stands for the \textit{plastic} region.)
\begin{defn}
\label{def: plastic}Let 
\begin{eqnarray*}
P^{+}:=\{x\in U:u(x)=\rho(x)\}, &  & P^{-}:=\{x\in U:u(x)=-\bar{\rho}(x)\}.
\end{eqnarray*}
Then $P:=P^{+}\cup P^{-}$ is called the \textbf{coincidence} set;
and 
\[
E:=\{x\in U:-\bar{\rho}(x)<u(x)<\rho(x)\}
\]
is called the \textbf{non-coincidence} set. We also define the \textbf{free
boundary} to be $\partial E\cap U$.
\end{defn}

\subsection{\label{subsec: Reg gaug}Regularity of the gauge function}

In this subsection we review some of the properties of $\gamma,\gamma^{\circ}$
and $K,K^{\circ}$ briefly. For detailed explanations and proofs see
\citep{MR3155183}. Recall that the gauge function $\gamma$ satisfies
\begin{align*}
 & \gamma(rx)=r\gamma(x),\\
 & \gamma(x+y)\le\gamma(x)+\gamma(y),
\end{align*}
for all $x,y\in\mathbb{R}^{n}$ and $r\ge0$. Also, note that as $B_{c}(0)\subseteq K\subseteq B_{C}(0)$
for some $C\ge c>0$, we have 
\[
\frac{1}{C}|x|\le\gamma(x)\le\frac{1}{c}|x|
\]
for all $x\in\mathbb{R}^{n}$.

It is well known that for all $x,y\in\mathbb{R}^{n}$ we have 
\begin{equation}
\langle x,y\rangle\leq\gamma(x)\gamma^{\circ}(y).\label{eq: gen Cauchy-Schwartz}
\end{equation}
In fact, more is true and we have 
\begin{equation}
\gamma^{\circ}(y)=\underset{x\ne0}{\max}\frac{\langle x,y\rangle}{\gamma(x)}.\label{eq: gen Cauchy-Schwartz 2}
\end{equation}
For a proof of this see page 54 of \citep{MR3155183}.

It is easy to see that the the strict convexity of $K$ (which means
that $\partial K$ does not contain any line segment) is equivalent
to the strict convexity of $\gamma$. By homogeneity of $\gamma$
the latter is equivalent to 
\[
\gamma(x+y)<\gamma(x)+\gamma(y)
\]
when $x\ne cy$ and $y\ne cx$ for any $c\ge0$.

Suppose that $\partial K$ is $C^{k,\alpha}$ $(k\ge1\,,\,0\le\alpha\le1)$.
Then $\gamma$ is $C^{k,\alpha}$ on $\mathbb{R}^{n}-\{0\}$ (see
for example \citep{SAFDARI202176}). Conversely, note that as $\partial K=\{\gamma=1\}$
and $D\gamma\ne0$ by (\ref{eq: g0 (Dg)=00003D1}), $\partial K$
is as smooth as $\gamma$. Suppose in addition that $K$ is strictly
convex. Then $\gamma$ is strictly convex too. By Remark 1.7.14 and
Theorem 2.2.4 of \citep{MR3155183}, $K^{\circ}$ is also strictly
convex and its boundary is $C^{1}$. Therefore $\gamma^{\circ}$ is
strictly convex, and it is $C^{1}$ on $\mathbb{R}^{n}-\{0\}$. Furthermore,
by Corollary 1.7.3 of \citep{MR3155183}, for $x\ne0$ we have 
\begin{eqnarray}
D\gamma(x)\in\partial K^{\circ}, &  & D\gamma^{\circ}(x)\in\partial K,\label{eq: g0 (Dg)=00003D1}
\end{eqnarray}
or equivalently 
\[
\gamma^{\circ}(D\gamma)=1,\qquad\gamma(D\gamma^{\circ})=1.
\]
In particular $D\gamma,D\gamma^{\circ}$ are nonzero on $\mathbb{R}^{n}-\{0\}$.

Now assume that $k\ge2$ and the principal curvatures of $\partial K$
are positive everywhere. Then $K$ is strictly convex. We can also
show that $\gamma^{\circ}$ is $C^{k,\alpha}$ on $\mathbb{R}^{n}-\{0\}$.
To see this let $n_{K}:\partial K\to\mathbb{S}^{n-1}$ be the Gauss
map, i.e. let $n_{K}(y)$ be the outward unit normal to $\partial K$
at $y$. Then $n_{K}$ is $C^{k-1,\alpha}$ and its derivative is
an isomorphism at the points with positive principal curvatures, i.e.
everywhere. Hence $n_{K}$ is locally invertible with a $C^{k-1,\alpha}$
inverse $n_{K}^{-1}$, around any point of $\mathbb{S}^{n-1}$. Now
note that as it is well known, $\gamma^{\circ}$ equals the support
function of $K$, i.e. 
\[
\gamma^{\circ}(x)=\sup\{\langle x,y\rangle:y\in K\}.
\]
Thus as shown in page 115 of \citep{MR3155183}, for $x\ne0$ we have
\[
D\gamma^{\circ}(x)=n_{K}^{-1}(\frac{x}{|x|}).
\]
Which gives the desired result. As a consequence, $\partial K^{\circ}$
is $C^{k,\alpha}$ too. Furthermore, as shown on page 120 of \citep{MR3155183},
the principal curvatures of $\partial K^{\circ}$ are also positive
everywhere.

Let us recall a few more properties of $\gamma,\gamma^{\circ}$. Since
they are positively 1-homogeneous, $D\gamma,D\gamma^{\circ}$ are
positively 0-homogeneous, and $D^{2}\gamma,D^{2}\gamma^{\circ}$ are
positively $(-1)$-homogeneous, i.e. 
\begin{eqnarray}
\gamma(tx)=t\gamma(x), & D\gamma(tx)=D\gamma(x), & D^{2}\gamma(tx)=\frac{1}{t}D^{2}\gamma(x),\nonumber \\
\gamma^{\circ}(tx)=t\gamma^{\circ}(x), & D\gamma^{\circ}(tx)=D\gamma^{\circ}(x), & D^{2}\gamma^{\circ}(tx)=\frac{1}{t}D^{2}\gamma^{\circ}(x),\label{eq: homog}
\end{eqnarray}
for $x\ne0$ and $t>0$. As a result, using Euler's theorem on homogeneous
functions we get 
\begin{eqnarray}
\langle D\gamma(x),x\rangle=\gamma(x), &  & D^{2}\gamma(x)\,x=0,\nonumber \\
\langle D\gamma^{\circ}(x),x\rangle=\gamma^{\circ}(x), &  & D^{2}\gamma^{\circ}(x)\,x=0,\label{eq: Euler formula}
\end{eqnarray}
for $x\ne0$. Here $D^{2}\gamma(x)\,x$ is the action of the matrix
$D^{2}\gamma(x)$ on the vector $x$. 

Finally let us mention that by Corollary 2.5.2 of \citep{MR3155183},
when $x\ne0$, the eigenvalues of $D^{2}\gamma(x)$ are $0$ with
the corresponding eigenvector $x$, and $\frac{1}{|x|}$ times the
principal radii of curvature of $\partial K^{\circ}$ at the unique
point that has $x$ as an outward normal vector. Remember that the
principal radii of curvature are the reciprocals of the principal
curvatures. Thus by our assumption, the eigenvalues of $D^{2}\gamma(x)$
are all positive except for one $0$. We have a similar characterization
of the eigenvalues of $D^{2}\gamma^{\circ}(x)$.

\subsection{\label{subsec: Reg Opstacles}Regularity of the obstacles}

Next let us consider the obstacles $\rho,-\bar{\rho}$, and review
some of their properties. All the results of this subsection are proved
in \citep{SAFDARI202176}.
\begin{defn}
When $\rho(x)=\gamma(x-y)+\varphi(y)$ for some $y\in\partial U$,
we call $y$ a \textbf{$\boldsymbol{\rho}$-closest} point to $x$
on $\partial U$. Similarly, when $\bar{\rho}(x)=\gamma(y-x)-\varphi(y)$
for some $y\in\partial U$, we call $y$ a \textbf{$\boldsymbol{\bar{\rho}}$-closest}
point to $x$ on $\partial U$.
\end{defn}
\begin{lem}
\label{lem: segment to the closest pt}Suppose $y$ is one of the
$\rho$-closest points on $\partial U$ to $x\in U$. Then 
\begin{enumerate}
\item[\upshape{(a)}] $y$ is a $\rho$-closest point on $\partial U$ to every point of
$]x,y[$. Therefore $\rho$ varies linearly along the line segment
$[x,y]$.
\item[\upshape{(b)}] If in addition, for all $x\ne y\in\R^{n}$ we have 
\begin{equation}
-\gamma(y-x)<\varphi(x)-\varphi(y)<\gamma(x-y),\label{eq: phi strct Lip}
\end{equation}
then we also have $]x,y[\subset U$.
\item[\upshape{(c)}] If in addition $\gamma$ is strictly convex, and the strict Lipschitz
property (\ref{eq: phi strct Lip}) for $\varphi$ holds, then $y$
is the unique $\rho$-closest point on $\partial U$ to the points
of $]x,y[$.
\end{enumerate}
\end{lem}
Next, we generalize the notion of ridge introduced by \citet{MR0184503},
and \citet{MR534111}. Intuitively, the $\rho$-ridge is the set of
singularities of $\rho$.
\begin{defn}
\label{def: ridge}The \textbf{$\boldsymbol{\rho}$-ridge} of $U$
is the set of all points $x\in U$ where $\rho(x)$ is not $C^{1,1}$
in any neighborhood of $x$. We denote it by 
\[
R_{\rho}.
\]
We have shown that when $\gamma$ is strictly convex and the strict
Lipschitz property (\ref{eq: phi strct Lip}) for $\varphi$ holds,
the points with more than one $\rho$-closest point on $\partial U$
belong to $\rho$-ridge, since $\rho$ is not differentiable at them.
This subset of the $\rho$-ridge is denoted by 
\[
R_{\rho,0}.
\]
Similarly we define $R_{\bar{\rho}},R_{\bar{\rho},0}$.%
\begin{comment}
\textendash{} mention other names of ridge, like cut locus ??

\textendash{} $R_{-\bar{\rho}}$ ??

\textendash{} $R_{\rho}$ is closed in $U$ ?
\end{comment}
\end{defn}
We know that $\rho,\bar{\rho}$ are Lipschitz functions. We want to
characterize the set over which they are more regular. In order to
do that, we need to impose some additional restrictions on $K,U$
and $\varphi$.
\begin{assumption}
\label{assu: K,U}Suppose that $k\ge2$ is an integer, and $0\le\alpha\le1$.
We assume that 
\begin{enumerate}
\item[\upshape{(a)}] $K\subset\R^{n}$ is a compact convex set whose interior contains
the origin. In addition, $\partial K$ is $C^{k,\alpha}$, and its
principal curvatures are positive at every point.
\item[\upshape{(b)}] $U\subset\R^{n}$ is a bounded open set, and $\partial U$ is $C^{k,\alpha}$.
\item[\upshape{(c)}] $\varphi:\R^{n}\to\R$ is a $C^{k,\alpha}$ function, such that $\gamma^{\circ}(D\varphi)<1$.
\end{enumerate}
\end{assumption}
\begin{rem*}
As shown in Subsection \ref{subsec: Reg gaug}, the above assumption
implies that $K,\gamma$ are strictly convex. In addition, $K^{\circ},\gamma^{\circ}$
are strictly convex, and $\partial K^{\circ},\gamma^{\circ}$ are
also $C^{k,\alpha}$. Furthermore, the principal curvatures of $\partial K^{\circ}$
are also positive at every point. Similar conclusions obviously hold
for $-K,-\varphi$ and $(-K)^{\circ}=-K^{\circ}$ too. Hence in the
sequel, whenever we prove a property for $\rho$, it holds for $\bar{\rho}$
too.\textcolor{red}{}
\end{rem*}
Let $\nu$ be the inward unit normal to $\partial U$. Then for every
$y\in\partial U$ there is a unique scalar $\lambda(y)>0$ such that
\begin{equation}
\gamma^{\circ}\big(D\varphi(y)+\lambda(y)\nu(y)\big)=1.\label{eq: lambda}
\end{equation}
We set 
\begin{equation}
\mu(y):=D\varphi(y)+\lambda(y)\nu(y).\label{eq: mu}
\end{equation}
We also set 
\begin{equation}
X:=\frac{1}{\langle D\gamma^{\circ}(\mu),\nu\rangle}D\gamma^{\circ}(\mu)\otimes\nu,\label{eq: X}
\end{equation}
where $a\otimes b$ is the rank 1 matrix whose action on a vector
$z$ is $\langle z,b\rangle a$. Let $x\in U$, and suppose $y$ is
one of the $\rho$-closest points to $x$ on $\partial U$. Then we
have 
\begin{equation}
\frac{x-y}{\gamma(x-y)}=D\gamma^{\circ}(\mu(y)).\label{eq: K-normal}
\end{equation}
Or equivalently 
\begin{equation}
x=y+\big(\rho(x)-\varphi(y)\big)\,D\gamma^{\circ}(\mu(y)).\label{eq: parametrize by rho}
\end{equation}
Also, $\rho$ is differentiable at $x$ if and only if $x\in U-R_{\rho,0}$.
And in that case we have 
\begin{equation}
D\rho(x)=\mu(y),\label{eq: D rho (x)}
\end{equation}
where $y$ is the unique $\rho$-closest point to $x$ on $\partial U$.

In addition, for every $y\in\partial U$ there is an open ball $B_{r}(y)$
such that $\rho$ is $C^{k,\alpha}$ on $\overline{U}\cap B_{r}(y)$.
Furthermore, $y$ is the $\rho$-closest point to some points in $U$,
and we have 
\begin{equation}
D\rho(y)=\mu(y).\label{eq: D rho (y)}
\end{equation}
We also have 
\begin{equation}
D^{2}\rho(y)=(I-X^{T})\big(D^{2}\varphi(y)+\lambda(y)D^{2}d(y)\big)(I-X),\label{eq: D2 rho (y)}
\end{equation}
where $I$ is the identity matrix, $d$ is the Euclidean distance
to $\partial U$, and $X$ is given by (\ref{eq: X}).
\begin{rem*}
As a consequence, $R_{\rho}$ has a positive distance from $\partial U$.
\end{rem*}
Let $x\in U-R_{\rho,0}$, and let $y$ be the unique $\rho$-closest
point to $x$ on $\partial U$. Let 
\begin{align}
 & W=W(y):=-D^{2}\gamma^{\circ}(\mu(y))D^{2}\rho(y),\nonumber \\
 & Q=Q(x):=I-\big(\rho(x)-\varphi(y)\big)W,\label{eq: W,Q}
\end{align}
where $I$ is the identity matrix. If $\det Q\ne0$ then $\rho$ is
$C^{k,\alpha}$ on a neighborhood of $x$. In addition we have 
\begin{equation}
D^{2}\rho(x)=D^{2}\rho(y)Q(x)^{-1}.\label{eq: D2 rho (x)}
\end{equation}
We also have 
\begin{equation}
x\in R_{\rho}\textrm{ if and only if }\det Q(x)=0.\label{eq: ridge =00003D Q=00003D0}
\end{equation}

\begin{rem*}
When $\varphi=0$, the function $\rho$ is the distance to $\partial U$
with respect to the Minkowski distance defined by $\gamma$. So this
case has a geometric interpretation. An interesting fact is that in
this case the eigenvalues of $W$ coincide with the notion of curvature
of $\partial U$ with respect to some Finsler structure. For the details
see \citep{MR2336304}.
\end{rem*}
\begin{lem}
\label{lem: D2 rho decreas}Suppose the Assumption \ref{assu: K,U}
holds. Let $x\in U-R_{\rho}$, and let $y$ be the unique $\rho$-closest
point to $x$ on $\partial U$. Then we have 
\[
D_{\xi\xi}^{2}\rho(x)\le D_{\xi\xi}^{2}\rho(y)
\]
for every $\xi\in\R^{n}$.
\end{lem}
As we mentioned in the introduction, the above monotonicity property
is true because $\rho$ satisfies the Hamilton-Jacobi equation (\ref{eq: H-J eq}),
and  the segment $]x,y[$ is the characteristic curve associated
to it. Let us review the general case of the monotonicity property
below. (Although note that the assumptions in Lemma \ref{lem: D2 rho decreas}
are weaker than what we assume in the following calculations. For
a proof of Lemma \ref{lem: D2 rho decreas} see the proof of Lemma
4 in \citep{SAFDARI202176}.)%
\begin{comment}
UPDATE the lemma no. after publication ??
\end{comment}

\textbf{Monotonicity of the second derivative of the solutions to
Hamilton-Jacobi equations:} Suppose $v$ satisfies the equation $\tilde{H}(Dv)=0$,
where $\tilde{H}$ is a convex function. Let $x(s)$ be a characteristic
curve of the equation. Then we have $\dot{x}=D\tilde{H}$. Let us
assume that $v$ is $C^{3}$ on a neighborhood of the image of $x(s)$.
Let 
\[
q(s):=D_{\xi\xi}^{2}v(x(s))=\xi_{i}\xi_{j}D_{ij}^{2}v
\]
for some vector $\xi$. Then we have 
\[
\dot{q}=\xi_{i}\xi_{j}D_{ijk}^{3}v\,\dot{x}^{k}=\xi_{i}\xi_{j}D_{ijk}^{3}vD_{k}\tilde{H}.
\]
On the other hand, if we differentiate the equation we get $D_{k}\tilde{H}D_{ik}^{2}v=0$.
And if we differentiate one more time we get 
\[
D_{kl}^{2}\tilde{H}D_{jl}^{2}vD_{ik}^{2}v+D_{k}\tilde{H}D_{ijk}^{3}v=0.
\]
Now if we multiply the above expression by $\xi_{i}\xi_{j}$, and
sum over $i,j$, we obtain the following Riccati type equation 
\begin{equation}
\dot{q}=-\xi^{T}D^{2}v\,D^{2}\tilde{H}\,D^{2}v\xi.\label{eq: ODE D2 rho}
\end{equation}
So $\dot{q}\le0$, since $\tilde{H}$ is convex. Thus we have 
\[
D_{\xi\xi}^{2}v(x(s))=q(s)\le q(0)=D_{\xi\xi}^{2}v(x(0)),
\]
as desired. This result also holds in the more general case of $\tilde{H}(x,v,Dv)=0$,
when $\tilde{H}$ is a convex function in all of its arguments (see
\citep{SAFDARI202176}).

\section{\label{sec: Proof Thm 3}Proof of Theorem \ref{thm: Reg dbl obstcl}}

In this section we prove Theorem \ref{thm: Reg dbl obstcl}, i.e.
we will prove that the double obstacle problem (\ref{eq: dbl obstcl})
has a solution $u$ in $W^{2,\infty}$, without assuming any regularity
about $K$. To this end first we need to prove Proposition \ref{prop: ridge is elastic},
which says that when $\partial K$ is smooth enough $u$ does not
touch the obstacles $\rho,-\bar{\rho}$ at their singularities. Before
that we need a few preliminary results. Throughout this section (before
the proof of Theorem \ref{thm: Reg dbl obstcl}) we assume that
\begin{enumerate}
\item $F$ does not depend on $x$, and satisfies Assumptions \ref{assu: F},\ref{assu: F C2}.
\item $\partial U$ is $C^{1}$.
\item $\varphi$ satisfies the Lipschitz property (\ref{eq: phi Lip}).
\end{enumerate}
Remember that when $u$ is a solution of the double obstacle problem
(\ref{eq: dbl obstcl}), we denote its non-coincidence and coincidence
sets by $E,P^{\pm}$, respectively.
\begin{lem}
Let $u\in C^{1}(\overline{U})\cap W_{\mathrm{loc}}^{2,n}(U)\cap W_{\bar{\rho},\rho}(U)$
be a solution of the double obstacle problem (\ref{eq: dbl obstcl}).
Then we have 
\[
\gamma^{\circ}(Du)\le1.
\]
\end{lem}
\begin{proof}
The proof is exactly the same as the proof of Theorem \ref{thm: dbl obstcl =00003D PDE grad}
on page \pageref{proof Thm 2}. We only need to consider the set $\{-\bar{\rho}<u<\rho\}$
instead of $\{u<\rho\}$.
\end{proof}
\begin{lem}
\label{lem: segment is plastic}Let $u\in C^{1}(\overline{U})\cap W_{\mathrm{loc}}^{2,n}(U)\cap W_{\bar{\rho},\rho}(U)$
be a solution of the double obstacle problem (\ref{eq: dbl obstcl}).
If $x\in P^{+}$, and $y$ is a $\rho$-closest point on $\partial U$
to $x$ such that $[x,y[\subset U$, then we have $[x,y[\subset P^{+}$.
Similarly, if $x\in P^{-}$, and $y$ is a $\bar{\rho}$-closest point
on $\partial U$ to $x$ such that $[x,y[\subset U$, then we have
$[x,y[\subset P^{-}$.%
\begin{comment}
Is this still true when the constraint depends on $x$ (and may be
$u$) ?
\end{comment}
\end{lem}
\begin{rem*}
If the strict Lipschitz property (\ref{eq: phi strct Lip}) for $\varphi$
holds, then by Lemma \ref{lem: segment to the closest pt} we automatically
have $[x,y[\subset U$. However, this can hold in other cases too;
see for example the proof of Proposition \ref{prop: E , P}.
\end{rem*}
\begin{proof}
Suppose $x\in P^{-}$; the other case is similar. We have 
\[
u(x)=-\bar{\rho}(x)=-\gamma(y-x)+\varphi(y).
\]
Let $\tilde{v}:=u-(-\bar{\rho})\ge0$, and $\xi:=\frac{y-x}{\gamma(y-x)}=-\frac{x-y}{\bar{\gamma}(x-y)}$.
Then $\bar{\rho}$ varies linearly along the segment $]x,y[$, since
$y$ is a $\bar{\rho}$-closest point to the points of the segment.
So we have $D_{\xi}(-\bar{\rho})=D_{-\xi}\bar{\rho}=1$ along the
segment. Note that we do not assume the differentiability of $\bar{\rho}$;
and $D_{-\xi}\bar{\rho}$ is just the derivative of the restriction
of $\bar{\rho}$ to the segment $]x,y[$. Now since 
\[
D_{\xi}u=\langle Du,\xi\rangle\le\gamma^{\circ}(Du)\gamma(\xi)\le1,
\]
we have $D_{\xi}\tilde{v}\le0$ along $]x,y[$. Thus as $\tilde{v}(x)=\tilde{v}(y)=0$,
and $\tilde{v}$ is continuous on the closed segment $[x,y]$, we
must have $\tilde{v}\equiv0$ on $[x,y]$. Therefore $u=-\bar{\rho}$
along the segment as desired.
\end{proof}
\begin{prop}
\label{prop: ridge is elastic}Let $u\in C^{1}(\overline{U})\cap W_{\mathrm{loc}}^{2,n}(U)\cap W_{\bar{\rho},\rho}(U)$
be a solution of the double obstacle problem (\ref{eq: dbl obstcl}).
In addition suppose that the Assumption \ref{assu: K,U} holds, and
$u\in W_{\mathrm{loc}}^{2,\infty}(U)$. Then we have 
\begin{eqnarray*}
R_{\rho}\cap P^{+}=\emptyset, & \hspace{2cm} & R_{\bar{\rho}}\cap P^{-}=\emptyset.
\end{eqnarray*}
\end{prop}
\begin{proof}
Note that due to Assumption \ref{assu: K,U}, the strict Lipschitz
property (\ref{eq: phi strct Lip}) for $\varphi$ holds, and $\gamma$
is strictly convex. First let us show that $R_{\bar{\rho},0}\cap P^{-}=\emptyset$;
the other case is similar. Suppose to the contrary that $x\in R_{\bar{\rho},0}\cap P^{-}$.
Then there are at least two distinct points $y,z\in\partial U$ such
that 
\[
\bar{\rho}(x)=\gamma(y-x)-\varphi(y)=\gamma(z-x)-\varphi(z).
\]
Now by Lemma \ref{lem: segment is plastic}, we have $[x,y[,[x,z[\subset P^{-}$.
In other words, $u=-\bar{\rho}$ on both of these segments. Therefore
by Lemma \ref{lem: segment to the closest pt}, $u$ varies linearly
on both of these segments. Hence we get 
\[
\big\langle Du(x),\frac{y-x}{\gamma(y-x)}\big\rangle=1=\big\langle Du(x),\frac{z-x}{\gamma(z-x)}\big\rangle.
\]
However since $\gamma$ is strictly convex, this contradicts the fact
that $\gamma^{\circ}(Du(x))\le1$.

So we only need to show that $R_{\rho}-R_{\rho,0},R_{\bar{\rho}}-R_{\bar{\rho},0}$
do not intersect $P^{+},P^{-}$ respectively. Suppose to the contrary
that there is a point $x\in U$ which belongs to $(R_{\rho}-R_{\rho,0})\cap P^{+}$;
the other case is similar. Let $y$ be the unique $\rho$-closest
point to $x$ on $\partial U$. Then we must have $\det Q(x)=0$,
where $Q$ is given by (\ref{eq: W,Q}). Let $z\in]x,y[$. Then by
Lemma \ref{lem: segment to the closest pt} we have $z\in U$, and
$y$ is the unique $\rho$-closest point on $\partial U$ to $z$.
In addition, as proved in \citep{SAFDARI202176}, we have $\det Q(z)\ne0$.
Hence $\rho$ is $C^{k,\alpha}$ on a neighborhood of the line segment
$]x,y[$. We call this neighborhood $V$. In the proof of Theorem
4%
\begin{comment}
UPDATE this after publication ??
\end{comment}
{} of \citep{SAFDARI202176} it has been shown that there is a vector
$\xi$ with $|\xi|=1$, which is not parallel to the segment $]x,y[$,
such that 
\begin{equation}
D_{\xi\xi}^{2}\rho(z)\to-\infty\qquad\textrm{ as }\;z\to x.\label{eq: D2 rho to - infty}
\end{equation}
Here $z$ converges to $x$ along the segment $]x,y[$.

Now since $x\in P^{+}$ we have $u(x)=\rho(x)$. Hence by lemma \ref{lem: segment is plastic}
we have $[x,y[\subset P^{+}$. Thus $u(z)=\rho(z)$ for every $z\in]x,y[$.
Also remember that $u\le\rho$ everywhere, since $u\in W_{\bar{\rho},\rho}$.
Hence $\rho-u$ is a $C^{1}$ function on $V$, which attains its
maximum, $0$, on $]x,y[$. Thus $Du=D\rho$ on the segment $]x,y[$.
Next we claim that for any $z\in]x,y[$ there are points $z_{i}:=z+\varepsilon_{i}\xi$
in $V$ converging to $z$ at which we have 
\[
D_{\xi}u(z_{i})\le D_{\xi}\rho(z_{i}).
\]
Since otherwise we would have $D_{\xi}u>D_{\xi}\rho$ on a segment
of the form $]z,z+r\xi[$ for some small $r>0$. But as $u(z)=\rho(z)$
and $Du(z)=D\rho(z)$, this implies that $u>\rho$ on $]z,z+r\xi[$;
which is a contradiction. Thus we get the desired. As a consequence
we have 
\[
D_{\xi}u(z_{i})-D_{\xi}u(z)\le D_{\xi}\rho(z_{i})-D_{\xi}\rho(z).
\]
By applying the mean value theorem to the restriction of $\rho$ to
the segment $[z,z_{i}]$ we get 
\begin{equation}
D_{\xi}u(z_{i})-D_{\xi}u(z)\le|z_{i}-z|D_{\xi\xi}^{2}\rho(w_{i})\label{eq: upper bd on D2u}
\end{equation}
for some $w_{i}\in]z,z_{i}[$. 

On the other hand, $u$ is a $W^{2,\infty}$ function on a neighborhood
of $x$ by our assumption. Consequently there is $C>0$ such that
\begin{equation}
-C\le\frac{D_{\xi}u(z_{i})-D_{\xi}u(z)}{|z_{i}-z|}\label{eq: lower bd on D2u}
\end{equation}
for distinct $z,z_{i}$ sufficiently close to $x$. Now let $z\in]x,y[$
be close enough to $x$ so that $D_{\xi\xi}^{2}\rho(z)<-3C$, which
is possible due to (\ref{eq: D2 rho to - infty}). Then let $z_{i}=z+\varepsilon_{i}\xi$
be close enough to $z$ so that we have $D_{\xi\xi}^{2}\rho(w_{i})<-2C$,
which is possible due to the continuity of $D^{2}\rho$ on $V$. But
this is in contradiction with (\ref{eq: upper bd on D2u}) and (\ref{eq: lower bd on D2u}).
\end{proof}
We do not use the next proposition directly in the proof of Theorem
\ref{thm: Reg dbl obstcl}, however, it completes our understanding
of the relation between double obstacle problems and gradient constraints.
The proposition says that $u$ hits the gradient constraint, i.e.
$H(Du)=0$, exactly when it hits one of the obstacles $-\bar{\rho},\rho$.%
\begin{comment}
In fact, the potential failure of this proposition prevents us from
applying the method of this paper to the $x$-dependence case of $F$
for gradient constraints ??
\end{comment}

\begin{prop}
\label{prop: E , P}Let $u\in C^{1}(\overline{U})\cap W_{\mathrm{loc}}^{2,n}(U)\cap W_{\bar{\rho},\rho}(U)$
be a solution of the double obstacle problem (\ref{eq: dbl obstcl}).
In addition suppose that $\gamma$ is strictly convex. Then we have
\[
P=\{x\in U:H(Du(x))=0\},\qquad E=\{x\in U:H(Du(x))<0\}.
\]
\end{prop}
\begin{proof}
First suppose $x\in P^{-}$; the case of $P^{+}$ is similar. Then
we have 
\[
u(x)=-\bar{\rho}(x)=-\gamma(y-x)+\varphi(y)
\]
for some $y\in\partial U$. Note that the set of $y$'s for which
this happens is closed; hence it is compact, since $U$ is bounded.
Let $y$ be a $\bar{\rho}$-closest point on $\partial U$ to $x$
which has the least Euclidean distance to $x$ among its $\bar{\rho}$-closest
points. Then we must have $[x,y[\subset U$. Since otherwise $]x,y[$
would intersect $\partial U$ in a point $z$, and $z$ would be a
$\bar{\rho}$-closest point to $x$ which is closer to $x$ than $y$
(in the Euclidean distance). Because by the Lipschitz property (\ref{eq: phi Lip})
for $\varphi$ and the collinearity of $x,z,y$ we would have 
\[
\gamma(z-x)-\varphi(z)\le\gamma(z-x)+\gamma(y-z)-\varphi(y)=\gamma(y-x)-\varphi(y).
\]
Thus $[x,y[\subset U$, and by Lemma \ref{lem: segment is plastic},
$u=-\bar{\rho}$ along the segment $[x,y[$. We also know that $\bar{\rho}$
varies linearly along the segment $[x,y[$, since $y$ is a $\bar{\rho}$-closest
point to the points of the segment. Hence we have $D_{\xi}u(x)=1$
for $\xi:=\frac{y-x}{\gamma(y-x)}$. Therefore $\gamma^{\circ}(Du(x))$
cannot be less than $1$ due to the equation (\ref{eq: gen Cauchy-Schwartz 2}).
Thus $H(Du(x))=\gamma^{\circ}(Du(x))-1=0$.

Conversely, assume that $H(Du(x))=0$. Then $\gamma^{\circ}(Du(x))=1$.
Hence by (\ref{eq: gen Cauchy-Schwartz 2}) there is $\tilde{\xi}$
with $\gamma(\tilde{\xi})=1$ such that $D_{\tilde{\xi}}u(x)=1$.
Suppose to the contrary that $x\in E$, i.e. $-\bar{\rho}(x)<u(x)<\rho(x)$.
As shown in the poof of Theorem \ref{thm: dbl obstcl =00003D PDE grad},
we know that $D_{\tilde{\xi}}u$ is $C^{2,\alpha}$ in $E$ and satisfies
the elliptic equation 
\[
F_{z}[u]D_{\tilde{\xi}}u+F_{p_{i}}[u]D_{i}D_{\tilde{\xi}}u+F_{M_{ij}}[u]D_{ij}^{2}D_{\tilde{\xi}}u=0.
\]
On the other hand on $\overline{U}$ we have 
\[
D_{\tilde{\xi}}u=\langle Du,\tilde{\xi}\rangle\le\gamma^{\circ}(Du)\gamma(\tilde{\xi})\le1.
\]
Let $E_{1}$ be the connected component of $E$ that contains $x$.
Then the strong maximum principle implies that $D_{\tilde{\xi}}u\equiv1$
over $E_{1}$.

Now consider the line passing through $x$ in the $\tilde{\xi}$ direction,
and suppose it intersects $\partial E_{1}$ for the first time in
$y:=x-\tau\tilde{\xi}$ for some $\tau>0$. If $y\in\partial U$,
then for $t>0$ we have 
\[
\frac{d}{dt}[u(y+t\tilde{\xi})]=D_{\tilde{\xi}}u(y+t\tilde{\xi})=1=\frac{d}{dt}[t\gamma(\tilde{\xi})]=\frac{d}{dt}[\gamma(y+t\tilde{\xi}-y)].
\]
Thus as $u(y)=\varphi(y)$ we get $u(x)=u(y+\tau\tilde{\xi})=\gamma(x-y)+\varphi(y)\ge\rho(x)$,
which is a contradiction. Now if $y\in U$, then as it also belongs
to $\partial E$ we have $y\in P$. If $u(y)=\rho(y)=\gamma(y-\tilde{y})+\varphi(\tilde{y})$
for some $\tilde{y}\in\partial U$, similarly to the above we obtain
\begin{align*}
u(x) & =\gamma(x-y)+u(y)\\
 & =\gamma(x-y)+\gamma(y-\tilde{y})+\varphi(\tilde{y})\ge\gamma(x-\tilde{y})+\varphi(\tilde{y})\ge\rho(x),
\end{align*}
which is again a contradiction.

On the other hand, suppose $u(y)=-\bar{\rho}(y)=-\gamma(\tilde{y}-y)+\varphi(\tilde{y})$
for some $\tilde{y}\in\partial U$. Let $\tilde{y}$ be the closest
point (in the Euclidean distance) to $y$ that satisfies this equation.
Then as we showed in the beginning of this proof we have $[y,\tilde{y}[\subset U$.
Hence by Lemma \ref{lem: segment is plastic} we have $u=-\bar{\rho}$
on the segment $[y,\tilde{y}[$; and consequently $D_{\hat{\xi}}u(y)=1$,
where $\hat{\xi}:=\frac{\tilde{y}-y}{\gamma(\tilde{y}-y)}$. On the
other hand we also have $D_{\tilde{\xi}}u(y)=1$, because $Du$ is
continuous on $\overline{U}$. Now since $\gamma$ is strictly convex,
and $\gamma^{\circ}(Du)\le1$, we must have $\tilde{\xi}=\hat{\xi}$.
Therefore $x,y,\tilde{y}$ are collinear, and $x,\tilde{y}$ are on
the same side of $y$. But $\tilde{y}$ cannot belong to $]y,x[\subset E_{1}\subset E\subset U$.
Hence we must have $x\in]y,\tilde{y}[\subset P^{-}$, which means
$u(x)=-\bar{\rho}(x)$; and this is a contradiction.
\end{proof}

Before presenting the proof of Theorem \ref{thm: Reg dbl obstcl},
let us note that we can weaken the requirement that $\gamma^{\circ}(D\varphi)<1$,
and under suitable conditions allow $\gamma^{\circ}(D\varphi)$ to
be equal to $1$ at some points. First let us review some well-known
facts from convex analysis. Consider a compact convex set $K$. Let
$x\in\partial K$ and $\mathrm{v}\in\R^{n}-\{0\}$. We say the hyperplane
\begin{equation}
\Gamma_{x,\mathrm{v}}:=\{y\in\R^{n}:\langle y-x,\mathrm{v}\rangle=0\}\label{eq: hyperplane}
\end{equation}
is a \textit{supporting hyperplane} of $K$ at $x$ if $K\subset\{y:\langle y-x,\mathrm{v}\rangle\le0\}$.
In this case we say $\mathrm{v}$ is an \textit{outer normal vector}
of $K$ at $x$. The \textit{normal cone} of $K$ at $x$ is the closed
convex cone 
\begin{equation}
N(K,x):=\{0\}\cup\{\mathrm{v}\in\mathbb{R}^{n}-\{0\}:\mathrm{v}\textrm{ is an outer normal vector of }K\textrm{ at }x\}.\label{eq: normal cone}
\end{equation}
It is easy to see that when $\partial K$ is $C^{1}$ we have 
\[
N(K,x)=\{tD\gamma(x):t\ge0\}.
\]
For more details see {[}\citealp{MR3155183}, Sections 1.3 and 2.2{]}.
Now we can state the more relaxed condition which can be used instead
of $\gamma^{\circ}(D\varphi)<1$:\medskip{}

\begin{enumerate}
\item[\upshape{($\ast$)}] \label{enu: cond *}$\gamma^{\circ}(D\varphi)\le1$ everywhere; and
if for some $y\in\partial U$ we have $\gamma^{\circ}(D\varphi(y))=1$
then we must have 
\[
\langle\mathrm{v},\nu(y)\rangle\ne0
\]
for every nonzero $\mathrm{v}\in N(K^{\circ},D\varphi(y))$.%
\begin{comment}
\textendash{} with $\gamma^{\circ}(\mathrm{v})=1$.

\textendash{} is it always true that $\langle\mathrm{v},\nu(y)\rangle>0$
??? (The last point here hints that NO) ?? Note that $\mu$ is not
the same as $D\varphi$ even (necessarily ?) in this case! So there
is no contradiction with the sign ??

\textendash{} this assumption is sufficient but not necessary. For
example when $\partial K^{\circ}$ has a flat face, and $\langle\mathrm{v}_{j},\nu(y_{j})\rangle\to>0$
($\lambda^{*}>0$ too). Or when $N(K^{\circ},D\varphi(y))$ had dimension
greater than one ??

\textendash{} Note that by this assumption, $\langle\mathrm{v}_{j},\nu(y_{j})\rangle\to\langle\mathrm{v},\nu(y)\rangle$
only for one of the cases $K^{\circ},-K^{\circ}$. So there is no
contradiction with the sign. (more explanation?) ??
\end{comment}
\medskip{}
\end{enumerate}
Let us further elaborate on the above condition and present a geometric
interpretation for it. As we have seen in Subsection \ref{subsec: Reg Opstacles},
there is $\lambda>0$ such that $\mu:=D\varphi+\lambda\nu$ satisfies
$\gamma^{\circ}(\mu)=1$. In addition, for a point $y\in\partial U$,
$D\gamma^{\circ}(\mu)$ is the direction along which lie the points
in $U$ that have $y$ as their $\rho$-closest point. Note that
we also have $D\gamma^{\circ}(\mu)\in N(K^{\circ},\mu)$. Now when
$\gamma^{\circ}(D\varphi)=1$, $D\varphi$ plays the role of $\mu$.
And $\mathrm{v}\in N(K^{\circ},D\varphi)$ plays the role of $D\gamma^{\circ}(\mu)$.
Hence we need to impose the condition $(\ast)$ in order to be sure
that there is a direction along which we can enter $U$ and hit the
points whose $\rho$-closest point is $y$.

In addition, note that the condition $(\ast)$ also holds when we
replace $K,\varphi,K^{\circ}$ by $-K,-\varphi$ and $(-K)^{\circ}=-K^{\circ}$.
In particular notice that if $D\varphi\in\partial K^{\circ}$, i.e.
if $\gamma^{\circ}(D\varphi(y))=1$, then we have $-D\varphi\in-\partial K^{\circ}=\partial(-K^{\circ})$;
and vice versa. In addition, it is easy to see that 
\[
\mathrm{v}\in N(K^{\circ},D\varphi(y))\iff-\mathrm{v}\in N(-K^{\circ},-D\varphi(y)).
\]
So as a result, $\rho,\bar{\rho}$ will have the same properties.%
\begin{comment}
As we will see in the following proof, in order to show that $u\in W^{2,p}(U)\cap W_{\mathrm{loc}}^{2,\infty}(U)$
for every $p<\infty$, we only need $\partial U,\varphi$ to be $C^{2}$.
But we need their $C^{2,\alpha}$ regularity to be able to apply the
result of \citep{indrei2016nontransversal}, and conclude the optimal
regularity of $u$ up to the boundary. (SINCE we used existence of
$C^{2,\alpha}$ solutions of fully nonlinear equations to show the
existence of $u_{k}$, this does not work ?? But we can approximate
$\varphi$ with $\varphi_{k}$, though probably not $\partial U$
??)
\end{comment}

\begin{proof}[\textbf{Proof of Theorem \ref{thm: Reg dbl obstcl}}]
\label{proof Thm 3} As it is well known, a compact convex set with
nonempty interior can be approximated, in the Hausdorff metric, by
a shrinking sequence of compact convex sets with nonempty interior
which have smooth boundaries with positive curvature (see for example
\citep{schmuckenschlaeger1993simple}). We apply this result to $K^{\circ}$.
Thus there is a sequence $K_{k}^{\circ}$ of compact convex sets,
that have smooth boundaries with positive curvature, and 
\begin{eqnarray*}
K_{k+1}^{\circ}\subset\mathrm{int}(K_{k}^{\circ}), & \qquad & K^{\circ}={\textstyle \bigcap}K_{k}^{\circ}.
\end{eqnarray*}
Notice that we can take the approximations of $K^{\circ}$ to be
the polar of other convex sets, because the double polar of a compact
convex set with $0$ in its interior is itself. Also note that $K_{k}$'s
are strictly convex compact sets with $0$ in their interior, which
have smooth boundaries with positive curvature. Furthermore we have
$K=(K^{\circ})^{\circ}\supset K_{k+1}\supset K_{k}$. For the proof
of these facts see {[}\citealp{MR3155183}, Sections 1.6, 1.7 and
2.5{]}.

To simplify the notation we use $\gamma_{k},\gamma_{k}^{\circ},\rho_{k},\bar{\rho}_{k}$
instead of $\gamma_{K_{k}},\gamma_{K_{k}^{\circ}},\rho_{K_{k},\varphi},\bar{\rho}_{K_{k},\varphi}$,
respectively. Also, let $R_{k}$ be the $\rho_{k}$-ridge, and let
$E_{k},P_{k}^{\pm}$ be the non-coincidence and coincidence sets of
$u_{k}$. Note that $K_{k},U,\varphi$ satisfy the Assumption \ref{assu: K,U}.
In particular we have $\gamma_{k}^{\circ}(D\varphi)<1$, since $D\varphi\in K^{\circ}\subset\mathrm{int}(K_{k}^{\circ})$.
Hence as we have shown in \citep{SAFDARI202176}, $\rho_{k},\bar{\rho}_{k}$
satisfy Assumption \ref{assu: =00005Cpsi +-} (stated in the appendix).
In addition, they are $C^{2,\alpha}$ on a neighborhood of $\partial U$.
Thus by Theorem \ref{thm: Reg u} in the appendix, there is $u_{k}\in W_{\bar{\rho}_{k},\rho_{k}}(U)\cap W^{2,\infty}(U)$
that satisfies the double obstacle problem 
\[
\begin{cases}
F[u_{k}]=0 & \textrm{ a.e. in }\{-\bar{\rho}_{k}<u_{k}<\rho_{k}\},\\
F[u_{k}]\le0 & \textrm{ a.e. on }\{u_{k}=\rho_{k}\},\\
F[u_{k}]\ge0 & \textrm{ a.e. on }\{u_{k}=-\bar{\rho}_{k}\}.
\end{cases}
\]
Therefore the lemmas and propositions of this section, especially
Proposition \ref{prop: ridge is elastic}, hold for each $u_{k}$.
(This is our only use of the assumptions that $F$ is $C^{2}$ and
does not depend on $x$. In the rest of the proof we do not use these
assumptions directly.) Also we know that 
\begin{equation}
-\bar{\rho}_{1}\le-\bar{\rho}_{k}\le u_{k}\le\rho_{k}\le\rho_{1}.\label{eq: u_k bdd}
\end{equation}
Note that $\rho_{k}\le\rho_{1}$ and $\bar{\rho}_{k}\le\bar{\rho}_{1}$,
since $\gamma_{k}\le\gamma_{1}$ due to $K_{k}\supset K_{1}$. 

We divide the rest of this proof into four parts. In Part I we derive
the uniform bound (\ref{eq: F=00005Bu=00005D bdd}), i.e. we show
that $F[u_{k}]$ is bounded independently of $k$. This is possible
mainly for two reasons. First we will use the facts that $D^{2}\rho_{k}$
attains its maximum on $\partial U$, and $u_{k}$ does not touch
$\rho_{k}$ at its singularities; so we get a one-way bound for $F[u_{k}]$
(and also for $D^{2}u_{k}$). For the other bound, we use the fact
that $u_{k}$ is a subsolution or a supersolution of $F=0$ over $P_{k}^{\pm}$.
Then in Part II we show that a subsequence of $u_{k}$ converges to
a function $u$ in $W^{2,p}$, which is a solution of the double obstacle
problem (\ref{eq: dbl obstcl}). Here we use the bound (\ref{eq: F=00005Bu=00005D bdd}),
obtained in Part I, to show that the $W^{2,p}$ norm of $u_{k}$ is
uniformly bounded. Then we extract a weakly convergent subsequence
of $u_{k}$, and show that its limit is a viscosity solution, and
hence a strong solution, of (\ref{eq: dbl obstcl}).

Next in Part III we show that $u$ is in $W_{\mathrm{loc}}^{2,\infty}$.
Here we use the one-sided bound for $D^{2}u_{k}$, and boundedness
of $F[u_{k}]$ (both obtained in Part I), to show that $D^{2}u_{k}$
is uniformly bounded on $P_{k}$. Then it is shown that we can apply
the result of \citet{MR3198649} (generalized by \citet{Indrei-Minne})
to conclude that $D^{2}u_{k}$ is locally uniformly bounded, and to
show that in fact $u_{k}$ converges weakly star in $W_{\mathrm{loc}}^{2,\infty}$
to $u$. Finally in Part IV we prove that the regularity of $u$ holds
up to the boundary. Here we straighten the boundary, and similarly
to Part III we show that the convergence to $u$ is actually weakly
star in $W^{2,\infty}$.

\medskip{}

PART I:\medskip{}

Let us show that 
\begin{equation}
\|F[u_{k}]\|_{L^{\infty}(U)}=\|F(u_{k},Du_{k},D^{2}u_{k})\|_{L^{\infty}(U)}\le C\label{eq: F=00005Bu=00005D bdd}
\end{equation}
for some $C$ independent of $k$. To see this note that on $E_{k}$
we have $F[u_{k}]=0$. So the desired bound trivially holds on $E_{k}$.
Next consider $P_{k}^{+}$. We have 
\[
F[u_{k}]\le0\qquad\textrm{ a.e. on }P_{k}^{+}.
\]
Thus we have an upper bound for $F[u_{k}]$ on $P_{k}^{+}$, independently
of $k$.

On the other hand, since $P_{k}^{+}$ does not intersect $R_{k}$
due to Proposition \ref{prop: ridge is elastic}, $\rho_{k}$ is at
least $C^{2}$ on $P_{k}^{+}$.%
\begin{comment}
WE only need $x$-independence here (and a similar use in part III)
???
\end{comment}
{} Now as $u_{k}=\rho_{k}$ on $P_{k}^{+}$,  for a.e. $x\in P_{k}^{+}$
we have $Du_{k}(x)=D\rho_{k}(x)$ and $D^{2}u_{k}(x)=D^{2}\rho_{k}(x)$.
On the other hand, by Lemma \ref{lem: D2 rho decreas} we know that
$D^{2}\rho_{k}(x)\le D^{2}\rho_{k}(y)$, where $y$ is the $\rho_{k}$-closest
point on $\partial U$ to $x$. Hence by the ellipticity of $F$ we
have 
\begin{align*}
F(u_{k}(x),Du_{k}(x),D^{2}u_{k}(x)) & =F(\rho_{k}(x),D\rho_{k}(x),D^{2}\rho_{k}(x))\\
 & \ge F(\rho_{k}(x),D\rho_{k}(x),D^{2}\rho_{k}(y)).
\end{align*}
Now note that $\rho_{k}$ is uniformly bounded due to (\ref{eq: u_k bdd}),
and $D\rho_{k}$ is uniformly bounded since $D\rho_{k}\in K_{k}^{\circ}\subset K_{1}^{\circ}$.
Thus in order to show that $F[u_{k}]$ has a uniform lower bound on
$P_{k}^{+}$, we only need to show that $D^{2}\rho_{k}$ is bounded
on $\partial U$ independently of $k$. This has been proved in the
proof of Theorem 5%
\begin{comment}
UPDATE this after publication ??
\end{comment}
{} of \citep{SAFDARI202176}. Here, the condition $\gamma^{\circ}(D\varphi)<1$,
or the more relaxed condition $(\ast)$ on page \pageref{enu: cond *},
is needed. Similarly, we can show that $F[u_{k}]$ is bounded on $P_{k}^{-}$,
independently of $k$. Hence we obtain the desired bound (\ref{eq: F=00005Bu=00005D bdd}).\medskip{}

PART II:\medskip{}

Now let $f_{k}:=F(u_{k},Du_{k},D^{2}u_{k})$. Then $u_{k}$ is a strong
solution to the fully nonlinear elliptic equation 
\[
F(u_{k},Du_{k},D^{2}u_{k})=f_{k},\qquad\qquad u_{k}|_{\partial U}=\varphi.
\]
Thus by $W^{2,p}$ estimates for fully nonlinear elliptic equations
(see for example Theorem 4.5 of \citep{winter2009w2}) we have 
\begin{equation}
\|u_{k}\|_{W^{2,p}(U)}\le C\big(\|f_{k}\|_{L^{p}(U)}+\|\varphi\|_{W^{2,\infty}(U)}+\|u_{k}\|_{L^{\infty}(U)}\big)\label{eq: u k in W2,p}
\end{equation}
for some constant $C$ independent of $k$.

Therefore $u_{k}$ is a bounded sequence in $W^{2,p}(U)$ due to (\ref{eq: F=00005Bu=00005D bdd})
and (\ref{eq: u_k bdd}). Consequently for every $\tilde{\alpha}<1$,
$\|u_{k}\|_{C^{1,\tilde{\alpha}}(\overline{U})}$ is bounded independently
of $k$, because $\partial U$ is $C^{2}$. Hence there is a subsequence
of $u_{k}$, which we still denote by $u_{k}$, that is strongly convergent
in $C^{1}(\overline{U})$, and weakly convergent in $W^{2,p}(U)$.
We call the limit $u$. Note that $u$ belongs to $W^{2,p}(U)$ for
every $p<\infty$. Furthermore we have $u\in W_{\bar{\rho},\rho}$
because of (\ref{eq: u_k bdd}), and the fact that $\rho_{k},\bar{\rho}_{k}$
uniformly converge to $\rho,\bar{\rho}$ respectively%
\begin{comment}
proof ??
\end{comment}
. Now note that $u_{k}$ is a strong solution of the equation 
\[
\max\{\min\{F[u_{k}],u_{k}+\bar{\rho}_{k}\},u_{k}-\rho_{k}\}=0.
\]
Hence $u_{k}$ is also a viscosity solution of the above equation
(see \citep{lions1983remark}). Therefore $u$ is a viscosity solution
of the equation 
\begin{equation}
\max\{\min\{F[u],u+\bar{\rho}\},u-\rho\}=0,\label{eq: eq instd of dbl obstcl}
\end{equation}
due to the stability of viscosity solutions (see \citep{crandall1992user}).

Let us show that $u$ is also a strong solution of the equation (\ref{eq: eq instd of dbl obstcl}).
We know that for a.e. $x_{0}\in U$ we have 
\[
u(x_{0}+h)=u(x_{0})+\langle Du(x_{0}),h\rangle+\frac{1}{2}\langle D^{2}u(x_{0})h,h\rangle+o(|h|^{2})
\]
for small $h\in\mathbb{R}^{n}$ (see for example Proposition 2.2 of
\citep{caffarelli1996viscosity}). Now let 
\[
\phi(h)=u(x_{0})+\langle Du(x_{0}),h\rangle+\frac{1}{2}\langle(D^{2}u(x_{0})+\varepsilon I)h,h\rangle
\]
for some $\varepsilon>0$. Then $\phi$ is a $C^{2}$ function and
$u-\phi$ has a local maximum at $x_{0}\in U$. Hence at $x_{0}$
we must have 
\[
\max\{\min\{F(u,D\phi,D^{2}\phi),u+\bar{\rho}\},u-\rho\}\le0.
\]
Thus at $x_{0}$ we have 
\[
\max\{\min\{F(u,Du,D^{2}u+\varepsilon I),u+\bar{\rho}\},u-\rho\}\le0.
\]
Therefore by sending $\varepsilon\to0$ we get $\max\{\min\{F[u],u+\bar{\rho}\},u-\rho\}\le0$
due to the continuity of $F$. Similarly we can show that $\max\{\min\{F[u],u+\bar{\rho}\},u-\rho\}\ge0$.
Thus $u$ is a strong solution of (\ref{eq: eq instd of dbl obstcl})
as desired. However, this means that $u$ satisfies the double obstacle
problem (\ref{eq: dbl obstcl}).

\medskip{}

PART III:\medskip{}

Next let us show that $u$ belongs to $W_{\mathrm{loc}}^{2,\infty}$.
First we need to prove that $D^{2}u_{k}$ is bounded on $P_{k}$ independently
of $k$. To see this consider $P_{k}^{+}$; the other case is similar.
We know that for a.e. $x\in P_{k}^{+}$ we have $D^{2}u_{k}(x)=D^{2}\rho_{k}(x)$
due to Proposition \ref{prop: ridge is elastic}. Also, as we mentioned
in Part I of the proof, $D^{2}\rho_{k}$ is bounded on $\partial U$
independently of $k$. Hence by Lemma \ref{lem: D2 rho decreas},
when $y$ is the $\rho_{k}$-closest point on $\partial U$ to $x\in P_{k}^{+}$
we have 
\begin{equation}
D^{2}u_{k}(x)=D^{2}\rho_{k}(x)\le D^{2}\rho_{k}(y)\le\tilde{C}I\label{eq: D2 uk is bdd above}
\end{equation}
for some $\tilde{C}$ independent of $k$. Thus $\tilde{C}I-D^{2}u_{k}\ge0$
a.e. on $P_{k}^{+}$. Therefore by the uniform ellipticity of $F$
we have 
\begin{align*}
-\Lambda\,\mathrm{tr}(\tilde{C}I-D^{2}u_{k})\le F(u_{k},Du_{k},D^{2}u_{k}+\tilde{C}I-D^{2}u_{k}) & -F(u_{k},Du_{k},D^{2}u_{k})\\
 & \qquad\le-\lambda\,\mathrm{tr}(\tilde{C}I-D^{2}u_{k}).
\end{align*}
However, we know that $F(u_{k},Du_{k},D^{2}u_{k})$ is uniformly bounded
due to (\ref{eq: F=00005Bu=00005D bdd}), and \linebreak{}
$F(u_{k},Du_{k},\tilde{C}I)$ is bounded due to the uniform boundedness
of $u_{k},Du_{k}$ (remember that $u_{k}$ is strongly convergent
in $C^{1}$). Therefore $\mathrm{tr}(\tilde{C}I-D^{2}u_{k})=n\tilde{C}-\Delta u_{k}$
is uniformly bounded. Now let $\xi,\xi_{1},\cdots,\xi_{n-1}$ be an
orthonormal basis of $\mathbb{R}^{n}$. Then by (\ref{eq: D2 uk is bdd above})
we have 
\[
D_{\xi\xi}^{2}u_{k}=\Delta u_{k}-\sum_{j\le n-1}D_{\xi_{j}\xi_{j}}^{2}u_{k}\ge\Delta u_{k}-(n-1)\tilde{C}.
\]
Hence $D^{2}u_{k}$ is also bounded below on $P_{k}^{+}$ independently
of $k$. The case of $P_{k}^{-}$ can be treated similarly.

Now let $x_{0}\in U$, and suppose that $B_{r}(x_{0})\subset U$.
Set $v_{k}(y):=u_{k}(x_{0}+ry)$ for $y\in B_{1}(0)$. Let 
\[
\tilde{F}(z,p,M):=F(z,\tfrac{1}{r}p,\tfrac{1}{r^{2}}M)-F(z,\tfrac{1}{r}p,0).
\]
Then by (\ref{eq: dbl obstcl}), and the arguments of the above paragraph,
we have 
\[
\begin{cases}
\tilde{F}(v_{k},Dv_{k},D^{2}v_{k})=\tilde{f}_{k} & \textrm{ a.e. in }B_{1}(0)\cap\Omega_{k},\\
|D^{2}v_{k}|\le C & \textrm{ a.e. in }B_{1}(0)-\Omega_{k},
\end{cases}
\]
for some $C$ independent of $k$. Here $\Omega_{k}:=\{y\in B_{1}(0):u_{k}(x_{0}+ry)\in E_{k}\}$,
and 
\[
\tilde{f}_{k}:=-F(v_{k},\tfrac{1}{r}Dv_{k},0).
\]
 Next recall that $\|u_{k}\|_{W^{2,n}(B_{r}(x_{0}))}$ is bounded
independently of $k$ due to (\ref{eq: u k in W2,p}),(\ref{eq: F=00005Bu=00005D bdd}).
Therefore $\|v_{k}\|_{W^{2,n}(B_{1}(0))}$ is bounded independently
of $k$ too. Also note that $\|\tilde{f}_{k}\|_{C^{\tilde{\alpha}}(B_{1}(0))}$
are bounded independently of $k$, since $\|u_{k}\|_{C^{1,\tilde{\alpha}}(\overline{U})}$
is bounded independently of $k$. Thus we can apply the result of
\citep{Indrei-Minne} to deduce that 
\[
|D^{2}v_{k}|\le\bar{C}\qquad\textrm{ a.e. in }B_{\frac{1}{2}}(0)
\]
for some $\bar{C}$ independent of $k$. Therefore 
\[
|D^{2}u_{k}|\le\tilde{C}\qquad\textrm{ a.e. in }B_{\frac{r}{2}}(x_{0})
\]
for some $\tilde{C}$ independent of $k$. Hence $u_{k}$ is a bounded
sequence in $W^{2,\infty}(B_{\frac{r}{2}}(x_{0}))$. Therefore a subsequence
of them converges weakly star in $W^{2,\infty}(B_{\frac{r}{2}}(x_{0}))$.
But the limit must be $u$; so we get $u\in W^{2,\infty}(B_{\frac{r}{2}}(x_{0}))$,
as desired.

\medskip{}

PART IV:\medskip{}

Finally let us show that $u$ belongs to $W^{2,\infty}(U)$. Let $x_{0}\in\partial U$.
Let $\Phi$ be a $C^{2,\alpha}$ change of coordinates on a neighborhood
of $x_{0}$ that flattens $\partial U$ around $x_{0}$. More specifically,
we assume that $\Phi:x\mapsto y$ maps a neighborhood of $x_{0}$
onto a neighborhood of $0$ that contains $\overline{B}_{1}(0)$,
and the $\Phi$-image of $U,\partial U$ lie respectively in the half-space
$\{y_{n}>0\}$ and on the plane $\{y_{n}=0\}$. Let $\Psi$ be the
inverse of $\Phi$. Then we have $y=\Phi(x)$ and $x=\Psi(y)$. Let
$B_{1}^{+}:=B_{1}(0)\cap\{y_{n}>0\}$ and $B_{1}':=B_{1}(0)\cap\{y_{n}=0\}$.
Now set 
\[
\hat{u}_{k}(y):=u_{k}(\Psi(y))-\varphi(\Psi(y))=u_{k}(x)-\varphi(x).
\]
It is obvious that $\hat{u}_{k}=0$ on $B_{1}'$. We also have $\hat{u}_{k}\in W^{2,n}(B_{1}^{+})\cap C^{1}(\overline{B}_{1}^{+})$
(see {[}\citealp{MR1814364}, Section 7.3{]}). In addition we have
\begin{align}
D\hat{u}_{k}(y) & =(Du_{k}(x)-D\varphi(x))D\Psi(y),\nonumber \\
D^{2}\hat{u}_{k}(y) & =(D^{2}u_{k}(x)-D^{2}\varphi(x))D\Psi(y)D\Psi(y)+(Du_{k}(x)-D\varphi(x))D^{2}\Psi(y).\label{eq: D u hat}
\end{align}
Therefore we get 
\[
\|\hat{u}_{k}\|_{W^{2,n}(B_{1}^{+})}\le C\big(\|u_{k}\|_{W^{2,n}(U)}+\|\varphi\|_{W^{2,\infty}(U)}\big)
\]
for some $C$ independent of $k$. Hence $\|\hat{u}_{k}\|_{W^{2,n}(B_{1}^{+})}$
is bounded independently of $k$, due to (\ref{eq: u k in W2,p}),(\ref{eq: F=00005Bu=00005D bdd}).

Now let 
\begin{align*}
\hat{F}(y,z,p,M) & :=F(z+\varphi,\;pD\Phi+D\varphi,\;MD\Phi D\Phi+D^{2}\varphi+pD^{2}\Phi)\\
 & \qquad\qquad\qquad\quad-F(z+\varphi,\;pD\Phi+D\varphi,\;D^{2}\varphi+pD^{2}\Phi),
\end{align*}
where $\varphi,\Phi$ are computed at $x=\Psi(y)$. Note that by differentiating
the equality $\Psi\circ\Phi=\mathrm{id}$ we get $D\Psi D\Phi=I$,
and $D\Psi D^{2}\Phi D\Psi+D^{2}\Psi D\Phi=0$. Hence by (\ref{eq: D u hat})
we can easily check that 
\begin{equation}
\hat{F}[\hat{u}_{k}]=F[u_{k}]-F(u_{k},Du_{k},\;D^{2}\varphi-(Du_{k}-D\varphi)D\Psi D^{2}\Phi).\label{eq: F hat (u hat)}
\end{equation}
It is also easy to see that $\hat{F}$ is uniformly elliptic, Holder
continuous, and convex in $M$; and satisfies $\hat{F}(y,z,p,0)=0$.

Let $\Omega_{k}:=\{y\in B_{1}^{+}:\Psi(y)\in E_{k}\}$. Then $D^{2}\hat{u}_{k}$
is bounded on $B_{1}^{+}-\Omega_{k}:=\{y\in B_{1}^{+}:\Psi(y)\in P_{k}\}$
independently of $k$ due to (\ref{eq: D u hat}); because $D^{2}u_{k}$
is bounded on $P_{k}$ independently of $k$, and $Du_{k}$ is bounded
independently of $k$. Therefore by (\ref{eq: dbl obstcl}) and (\ref{eq: F hat (u hat)})
we have 
\[
\begin{cases}
\hat{F}[\hat{u}_{k}]=\hat{f}_{k} & \textrm{ a.e. in }B_{1}^{+}\cap\Omega_{k},\\
|D^{2}\hat{u}_{k}|\le C & \textrm{ a.e. in }B_{1}^{+}-\Omega_{k},\\
\hat{u}_{k}=0 & \textrm{ on }B_{1}',
\end{cases}
\]
for some $C$ independent of $k$. Here 
\[
\hat{f}_{k}:=-F(u_{k},Du_{k},\;D^{2}\varphi-(Du_{k}-D\varphi)D\Psi D^{2}\Phi).
\]
Note that $\hat{f}_{k}\in C^{\alpha_{0}}(\overline{B}_{1}^{+})$ for
some $\alpha_{0}>0$, and $\|\hat{f}_{k}\|_{C^{\alpha_{0}}(\overline{B}_{1}^{+})}$
is bounded independently of $k$; since $\|u_{k}\|_{C^{1,\tilde{\alpha}}(\overline{U})}$
is bounded independently of $k$ for every $\tilde{\alpha}<1$. Hence
as shown in \citep{indrei2016nontransversal} we get 
\[
|D^{2}\hat{u}_{k}|\le\bar{C}\qquad\textrm{ a.e. in }B_{\frac{1}{2}}(0)\cap\{y_{n}>0\}
\]
for some $\bar{C}$ independent of $k$. Thus 
\[
|D^{2}u_{k}|\le\tilde{C}\qquad\textrm{ a.e. in }B_{r}(x_{0})\cap U
\]
for some $r>0$ and some $\tilde{C}$ independent of $k$; because
we can compute the derivatives of $u_{k}$ in terms of the derivatives
of $\hat{u}_{k}$ similarly to (\ref{eq: D u hat}).

Hence $u_{k}$ is a bounded sequence in $W^{2,\infty}(B_{r}(x_{0})\cap U)$.
Therefore a subsequence of them converges weakly star in $W^{2,\infty}(B_{r}(x_{0})\cap U)$.
But the limit must be $u$; so we get $u\in W^{2,\infty}(B_{r}(x_{0})\cap U)$.
Finally note that we can cover $\partial U$ with finitely many open
balls of the form $B_{r}(x_{0})$ for $x_{0}\in\partial U$, over
which $u$ is $W^{2,\infty}$. Also, there is an open subset of $U$
whose union with these balls cover $U$, and over it $u$ is $W^{2,\infty}$
too. Thus we can conclude that $u\in W^{2,\infty}(U)$, as desired.
\end{proof}

\appendix

\section{\label{sec: Appdx}Fully Nonlinear Double Obstacle Problems}

In this appendix we are going to study the general double obstacle
problem 
\begin{equation}
\begin{cases}
F[u]=0 & \textrm{ a.e. in }\{\psi^{-}<u<\psi^{+}\},\\
F[u]\le0 & \textrm{ a.e. on }\{u=\psi^{+}\},\\
F[u]\ge0 & \textrm{ a.e. on }\{u=\psi^{-}\},
\end{cases}\label{eq: dbl obstcl 2}
\end{equation}
where $u$ belongs to 
\[
W_{\psi^{\pm}}:=\{v\in W^{1,2}(U):\psi^{-}\le v\leq\psi^{+}\textrm{ a.e.}\}.
\]
Here we allow $F$ to also depend on $x$. We also let the obstacles
to be more general than $\rho,-\bar{\rho}$, but we require their
weak second derivatives to have one-sided bounds. We show that the
solution $u$ has the optimal $W^{2,\infty}$ regularity. This result
has been used in the proof of Theorem \ref{thm: Reg dbl obstcl}.
Most of the methods employed in this section are classical and well
known, but to the best of author's knowledge the results have not
appeared elsewhere. Especially since the results are about the double
obstacle problem, and there are far fewer works on this problem compared
to the obstacle problem. Nevertheless, we include the proofs here
for completeness.%
\begin{comment}
and determining the constants ?
\end{comment}
{} First let us state our assumptions about the obstacles $\psi^{\pm}$.
\begin{assumption}
\label{assu: =00005Cpsi +-}We assume that $\psi^{\pm}:\R^{n}\to\R$
are Lipschitz functions which satisfy\medskip{}
\begin{enumerate}
\item[\upshape{(a)}] For every $x,y\in\mathbb{R}^{n}$ we have 
\[
|\psi^{\pm}(x)-\psi^{\pm}(y)|\le C_{1}|x-y|.
\]
\item[\upshape{(b)}] $\psi^{+}=\psi^{-}$ on $\partial U$, and for all $x\notin\partial U$
we have 
\begin{equation}
0<\psi^{+}(x)-\psi^{-}(x)\le2C_{1}d(x),\label{eq: psi+ - psi- < 2d}
\end{equation}
where $d$ is the Euclidean distance to $\partial U$.
\item[\upshape{(c)}] We have 
\begin{equation}
\pm\mathfrak{D}_{h,\xi}^{2}\psi^{\pm}(x):=\pm\frac{\psi^{\pm}(x+h\xi)+\psi^{\pm}(x-h\xi)-2\psi^{\pm}(x)}{h^{2}}\leq\frac{C_{2}}{d(x)-h}\label{eq: bd D2 psi}
\end{equation}
for some $C_{2}>0$, and every nonzero $x,\xi\in\mathbb{R}^{n}$ with
$|\xi|\le1$, and every $0<h<d(x)$.
\end{enumerate}
\end{assumption}

\begin{rem*}
As we have seen in \citep{SAFDARI202176}, when $\partial K$ is $C^{2}$,
and $\varphi$ satisfies the strict Lipschitz property (\ref{eq: phi strct Lip}),
then $\rho,-\bar{\rho}$ satisfy the above assumption.
\end{rem*}
Let $\eta_{\varepsilon}$ be the standard mollifier. Then we define
\begin{align}
 & \psi_{\varepsilon}^{+}(x):=(\eta_{\varepsilon}*\psi^{+})(x):=\int_{|y|\leq\varepsilon}\eta_{\varepsilon}(y)\psi^{+}(x-y)\,dy,\nonumber \\
 & \psi_{\varepsilon}^{-}(x):=(\eta_{\varepsilon}*\psi^{-})(x)+\delta_{\varepsilon},\label{eq: psi _e}
\end{align}
where $3C_{1}\varepsilon<\delta_{\varepsilon}<4C_{1}\varepsilon$
is chosen such that $\partial\{\psi_{\varepsilon}^{-}<\psi_{\varepsilon}^{+}\}$
is $C^{\infty}$, which is possible by Sard's Theorem. Note that since
$\psi^{\pm}$ are defined on all of $\mathbb{R}^{n}$, $\psi_{\varepsilon}^{\pm}$
are smooth functions on $\mathbb{R}^{n}$. Also 
\[
|\psi_{\varepsilon}^{+}(x)-\psi^{+}(x)|\leq\int_{|y|\leq\varepsilon}\eta_{\varepsilon}(y)|\psi^{+}(x-y)-\psi^{+}(x)|\,dy\leq\int_{|y|\leq\varepsilon}C_{1}|y|\eta_{\varepsilon}(y)\,dy\le C_{1}\varepsilon.
\]
Similarly we have 
\[
2C_{1}\varepsilon<\psi_{\varepsilon}^{-}-\psi^{-}<5C_{1}\varepsilon.
\]

Now, let 
\begin{equation}
U_{\varepsilon}:=\{x\in U:\psi_{\varepsilon}^{-}(x)<\psi_{\varepsilon}^{+}(x)\}.\label{eq: U_e}
\end{equation}
Then we have 
\begin{align}
\{x\in U:\psi^{+}(x)-\psi^{-}(x) & >5C_{1}\varepsilon\}\subset U_{\varepsilon}\nonumber \\
 & \subset\{x\in\overline{U}:\psi_{\varepsilon}^{-}(x)\leq\psi_{\varepsilon}^{+}(x)\}\subset\{x\in U:d(x)>\varepsilon\}.\label{eq: inclusions}
\end{align}
To see this note that $\psi_{\varepsilon}^{-}(x)\leq\psi_{\varepsilon}^{+}(x)$
implies that 
\[
3C_{1}\varepsilon<\delta_{\varepsilon}\le(\psi^{+}-\psi^{-})\ast\eta_{\varepsilon}\le\psi^{+}-\psi^{-}+C_{1}\varepsilon\le2C_{1}d(x)+C_{1}\varepsilon.
\]
Hence $d(x)>\varepsilon$. On the other hand, if $\psi_{\varepsilon}^{-}(x)\ge\psi_{\varepsilon}^{+}(x)$
then 
\[
4C_{1}\varepsilon>\delta_{\varepsilon}\ge(\psi^{+}-\psi^{-})\ast\eta_{\varepsilon}\ge\psi^{+}-\psi^{-}-C_{1}\varepsilon.
\]
Thus $\psi^{+}(x)-\psi^{-}(x)<5C_{1}\varepsilon$. Hence $\psi^{+}(x)-\psi^{-}(x)>5C_{1}\varepsilon$
implies $\psi_{\varepsilon}^{-}(x)<\psi_{\varepsilon}^{+}(x)$, as
desired.
\begin{rem*}
The above inclusions show that $\overline{U}_{\varepsilon}\subset U$,
and 
\begin{equation}
U=\bigcup_{\varepsilon>0}U_{\varepsilon};\label{eq: U=00003DU U_e}
\end{equation}
since by (\ref{eq: psi+ - psi- < 2d}) we know that $\psi^{+}-\psi^{-}>0$
on $U$. In addition, remember that we have chosen $\delta_{\varepsilon}$
so that $\partial U_{\varepsilon}$ is $C^{\infty}$. Furthermore,
for every $\varepsilon$ there is $\tilde{\varepsilon}$ such that
\begin{equation}
U_{\varepsilon}\subset\{d>\varepsilon\}\subset\{\psi^{+}-\psi^{-}>5C_{1}\tilde{\varepsilon}\}\subset U_{\tilde{\varepsilon}}.\label{eq: U_e subst U_e'}
\end{equation}
Because otherwise for every $j$ there is $x_{j}\in U$ such that
$d(x_{j})>\varepsilon$, while $\psi^{+}(x_{j})-\psi^{-}(x_{j})\le\frac{1}{j}$.
But due to the compactness we can assume that $x_{j}\to x\in\overline{U}$.
Then by continuity we must have $\psi^{+}(x)-\psi^{-}(x)=0$ and $d(x)\ge\varepsilon$.
Now by (\ref{eq: psi+ - psi- < 2d}), $\psi^{+}(x)-\psi^{-}(x)=0$
implies that $x\in\partial U$, which contradicts the fact that $d(x)\ge\varepsilon$.
\end{rem*}
\begin{lem}
\label{lem: D, D2 psi_e}Suppose that Assumption \ref{assu: =00005Cpsi +-}
holds. Then we have 
\begin{equation}
|D\psi_{\varepsilon}^{\pm}|\le C_{1}.\label{eq: bd D psi _e}
\end{equation}
Furthermore, for any unit vector $\xi$, and every $x\in U$ with
$d(x)>\varepsilon$ we have 
\begin{equation}
\pm D_{\xi\xi}^{2}\psi_{\varepsilon}^{\pm}(x)\le\frac{C_{2}}{d(x)-\varepsilon},\label{eq: bd D2 psi _e}
\end{equation}
where $d$ is the Euclidean distance to $\partial U$.
\end{lem}
\begin{proof}
To show the first part, note that $\psi^{\pm}$ are Lipschitz functions
and $|D\psi^{\pm}|\le C_{1}$ a.e. Thus we have 
\begin{align*}
|D\psi_{\varepsilon}^{\pm}(x)| & \leq\int_{|y|\leq\varepsilon}|\eta_{\varepsilon}(y)D\psi^{\pm}(x-y)|\,dy\\
 & =\int_{|y|\leq\varepsilon}\eta_{\varepsilon}(y)|D\psi^{\pm}(x-y)|\,dy\leq C_{1}\int_{|y|\leq\varepsilon}\eta_{\varepsilon}(y)\,dy\;=C_{1}.
\end{align*}
Next, suppose $d(x)>h+\varepsilon$, and $|\xi|=1$. Then due to the
Lipschitz continuity of $d$, for $|y|\le\varepsilon$ we have 
\[
d(x-y)\geq d(x)-|y|\ge d(x)-\varepsilon>h.
\]
Hence by (\ref{eq: bd D2 psi}) we get 
\begin{align*}
\pm\mathfrak{D}_{h,\xi}^{2}\psi_{\varepsilon}^{\pm}(x) & =\pm\int_{|y|\le\varepsilon}\eta_{\varepsilon}(y)\mathfrak{D}_{h,\xi}^{2}\psi^{\pm}(x-y)\,dy\\
 & \leq\int_{|y|\le\varepsilon}\eta_{\varepsilon}(y)\frac{C_{2}}{d(x-y)-h}\,dy\\
 & \leq\int_{|y|\le\varepsilon}\eta_{\varepsilon}(y)\frac{C_{2}}{d(x)-\varepsilon-h}\,dy=\frac{C_{2}}{d(x)-\varepsilon-h}.
\end{align*}
Let $h\rightarrow0^{+}$. Then for $x\in U$ with $d(x)>\varepsilon$
we get 
\[
\pm D_{\xi\xi}^{2}\psi_{\varepsilon}^{\pm}(x)\leq\frac{C_{2}}{d(x)-\varepsilon},
\]
as desired.
\end{proof}

Now consider the double obstacle problem 
\begin{equation}
\begin{cases}
F[u_{\varepsilon}]=0 & \textrm{ a.e. in }\{\psi_{\varepsilon}^{-}<u_{\varepsilon}<\psi_{\varepsilon}^{+}\},\\
F[u_{\varepsilon}]\le0 & \textrm{ a.e. on }\{u_{\varepsilon}=\psi_{\varepsilon}^{+}\},\\
F[u_{\varepsilon}]\ge0 & \textrm{ a.e. on }\{u_{\varepsilon}=\psi_{\varepsilon}^{-}\},
\end{cases}\label{eq: dbl obstcl u_e}
\end{equation}
where $u_{\varepsilon}$ belongs to $W_{\psi_{\varepsilon}^{\pm}}:=\{v\in W^{1,2}(U_{\varepsilon}):\psi_{\varepsilon}^{-}\leq v\leq\psi_{\varepsilon}^{+}\textrm{ a.e.}\}$.
\begin{lem}
\label{lem: Reg u_e}Suppose $F$ satisfies Assumptions \ref{assu: F},\ref{assu: F Lip}.
Also, suppose $\psi^{\pm}$ satisfy Assumption \ref{assu: =00005Cpsi +-}.
Then the double obstacle problem (\ref{eq: dbl obstcl u_e}) has a
solution $u_{\varepsilon}$, and for every $p<\infty$ we have 
\[
u_{\varepsilon}\in W^{2,p}(U_{\varepsilon}).
\]
\end{lem}
\begin{proof}
Fix $\varepsilon>0$. For $\delta>0$, let $\beta_{\delta}$ be a
smooth increasing function that vanishes on $(-\infty,0]$, and equals
$\frac{1}{\delta}t$ for $t\ge\delta$. Then the equation 
\begin{equation}
\begin{cases}
F(x,u_{\varepsilon,\delta},Du_{\varepsilon,\delta},D^{2}u_{\varepsilon,\delta})-\beta_{\delta}(\psi_{\varepsilon}^{-}-u_{\varepsilon,\delta})+\beta_{\delta}(u_{\varepsilon,\delta}-\psi_{\varepsilon}^{+})=0,\\
u_{\varepsilon,\delta}=\psi_{\varepsilon}^{+}\textrm{ on }\partial U_{\varepsilon},
\end{cases}\label{eq: u-e,d}
\end{equation}
has a unique solution in $C^{2,\alpha}(\overline{U}_{\varepsilon})$
(see for example Theorem 7.4 of \citep{chen1998second}). To simplify
the notation we set 
\[
\tilde{u}=u_{\varepsilon,\delta},\qquad\qquad\beta=\beta_{\delta}.
\]
First let us show that $\tilde{u}$ is uniformly bounded independently
of $\delta$. Suppose $C^{+}$ is a positive constant larger than
the maximum of $|\psi_{\varepsilon}^{\pm}|+1$ on $\overline{U}_{\varepsilon}$.
Now if we apply the above differential operator to the constant function
whose value is $C^{+}$ we obtain 
\begin{align*}
F(x,C^{+},0,0)-\beta(\psi_{\varepsilon}^{-}-C^{+})\:+\: & \beta(C^{+}-\psi_{\varepsilon}^{+})\\
 & =F(x,C^{+},0,0)+\frac{C^{+}-\psi_{\varepsilon}^{+}}{\delta}.
\end{align*}
This last expression is positive for $\delta$ small enough, since
$F(x,C^{+},0,0)$ is bounded on $\overline{U}_{\varepsilon}$. Therefore
by the comparison principle we have $\tilde{u}\le C^{+}$. We can
similarly show that $\tilde{u}\ge-C^{+}$. Hence for small enough
$\delta$ we have 
\begin{equation}
-C^{+}\le\tilde{u}\le C^{+}.\label{eq: u_e,d bdd}
\end{equation}

Now let us show that 
\[
\|\beta(\pm(\tilde{u}-\psi_{\varepsilon}^{\pm}))\|_{L^{\infty}(U_{\varepsilon})}\le C,
\]
where $C$ is independent of $\delta$. Note that $\beta(\pm(\tilde{u}-\psi_{\varepsilon}^{\pm}))$
is zero on $\partial U_{\varepsilon}$. So assume that $\beta(\pm(\tilde{u}-\psi_{\varepsilon}^{\pm}))$
attains its positive maximum at $x_{0}\in U_{\varepsilon}$. Let us
consider $\beta(\tilde{u}-\psi_{\varepsilon}^{+})$; the other case
is similar. Since $\beta$ is increasing, $\tilde{u}-\psi_{\varepsilon}^{+}$
has a positive maximum at $x_{0}$ too. Therefore we have 
\[
D\tilde{u}(x_{0})=D\psi_{\varepsilon}^{+}(x_{0}),\qquad\qquad D^{2}\tilde{u}(x_{0})\le D^{2}\psi_{\varepsilon}^{+}(x_{0}).
\]
We also have $\tilde{u}(x_{0})>\psi_{\varepsilon}^{+}(x_{0})\ge\psi_{\varepsilon}^{-}(x_{0})$.
Hence by the ellipticity of $F$, and its monotonicity in $z$, at
$x_{0}$ we have 
\begin{align*}
F(x_{0},\psi_{\varepsilon}^{+},D\psi_{\varepsilon}^{+},D^{2}\psi_{\varepsilon}^{+}) & \le F(x_{0},\tilde{u},D\tilde{u},D^{2}\tilde{u})\\
 & =\beta(\psi_{\varepsilon}^{-}-\tilde{u})-\beta(\tilde{u}-\psi_{\varepsilon}^{+})=-\beta(\tilde{u}-\psi_{\varepsilon}^{+}).
\end{align*}
Thus $\beta(\tilde{u}-\psi_{\varepsilon}^{+})\le-F[\psi_{\varepsilon}^{+}]$
at $x_{0}$. Therefore $\beta(\tilde{u}-\psi_{\varepsilon}^{+})$
is bounded independently of $\delta$, as desired.

The bound $\beta(\pm(\tilde{u}-\psi_{\varepsilon}^{\pm}))\le C$,
and the definition of $\beta$ imply that 
\begin{equation}
\tilde{u}-\psi_{\varepsilon}^{+}\le\delta(C+1),\qquad\qquad\psi_{\varepsilon}^{-}-\tilde{u}\le\delta(C+1).\label{eq: 1 in Reg u_e}
\end{equation}
In addition, from the equation (\ref{eq: u-e,d}) we conclude that
\[
\|F[\tilde{u}]\|_{L^{\infty}(U_{\varepsilon})}\le2C.
\]
Thus by $W^{2,p}$ estimates for fully nonlinear elliptic equations
(see Theorem 4.5 of \citep{winter2009w2}) we have 
\begin{equation}
\|\tilde{u}\|_{W^{2,p}(U_{\varepsilon})}\le\bar{C}\big(\|F[\tilde{u}]\|_{L^{p}(U_{\varepsilon})}+\|\psi_{\varepsilon}^{+}\|_{C^{2}(\overline{U})}+\|\tilde{u}\|_{L^{\infty}(U)}\big)\label{eq: u_e,d in W2,p}
\end{equation}
for some constant $\bar{C}$ independent of $\delta$. We only need
to check that for a constant $\tilde{\beta}_{0}$, which is determined
by $F,p$, we have 
\[
\sup_{M}\frac{|F(x,0,0,M)-F(x_{0},0,0,M)|}{|M|+1}\le\tilde{\beta}_{0}
\]
whenever $x,x_{0}\in\overline{U}$ and $|x-x_{0}|$ is small enough.
However, this follows easily from our assumption about $|F_{x}|$.
\begin{comment}
(Note that the above condition is expressed in \citep{winter2009w2}
with $|M|$ instead of $|M|+1$. But it is well known that the above
version is sufficient in the proofs of $W^{2,p}$ estimates; see for
example \citep{Hynd,indrei2016nontransversal}.)
\end{comment}

Therefore $\tilde{u}$ is bounded in $W^{2,p}(U_{\varepsilon})$ independently
of $\delta$, due to the uniform boundedness of $\tilde{u},F[\tilde{u}]$.
Consequently for every $\tilde{\alpha}<1$, $\|\tilde{u}\|_{C^{1,\tilde{\alpha}}(\overline{U}_{\varepsilon})}$
is bounded independently of $\delta$, because $\partial U_{\varepsilon}$
is $C^{2}$. Hence there is a sequence $\delta_{j}\to0$ such that
$\tilde{u}_{j}:=u_{\varepsilon,\delta_{j}}$ is strongly convergent
in $C^{1}(\overline{U}_{\varepsilon})$, and weakly convergent in
$W^{2,p}(U_{\varepsilon})$. We denote this limit by $u_{\varepsilon}$.
Note that $u_{\varepsilon}\in W^{2,p}(U_{\varepsilon})$. Also note
that if we let $\delta_{j}\to0$ in (\ref{eq: 1 in Reg u_e}) we get
$\psi_{\varepsilon}^{-}\leq u_{\varepsilon}\leq\psi_{\varepsilon}^{+}$.

Finally, let us show that $u_{\varepsilon}$ satisfies the double
obstacle problem (\ref{eq: dbl obstcl u_e}). It suffices to show
that $u_{\varepsilon}$ satisfies 
\begin{equation}
\max\{\min\{F[u_{\varepsilon}],u_{\varepsilon}-\psi_{\varepsilon}^{-}\},u_{\varepsilon}-\psi_{\varepsilon}^{+}\}=0.\label{eq: eq instd of dbl obstcl u_e}
\end{equation}
First let us show that $u_{\varepsilon}$ is a viscosity solution
of the above equation. Suppose $\phi$ is a $C^{2}$ function and
$u_{\varepsilon}-\phi$ has a local maximum at $x_{0}\in U$. We can
assume that $u_{\varepsilon}-\phi$ has a strict local maximum at
$x_{0}$ without loss of generality, since we can approximate $\phi$
with $\phi+\epsilon|x-x_{0}|^{2}$. We must show that at $x_{0}$
we have 
\begin{equation}
\max\{\min\{F(x_{0},u_{\varepsilon},D\phi,D^{2}\phi),u_{\varepsilon}-\psi_{\varepsilon}^{-}\},u_{\varepsilon}-\psi_{\varepsilon}^{+}\}\le0.\label{eq: 2 in Reg u_e}
\end{equation}
Now we know that $\tilde{u}_{j}-\phi$ has a local maximum at a point
$x_{j}$ where $x_{j}\to x_{0}$; because $\tilde{u}_{j}$ uniformly
converges to $u_{\varepsilon}$. Hence we have 
\[
D\tilde{u}_{j}(x_{j})=D\phi(x_{j}),\qquad\qquad D^{2}\tilde{u}_{j}(x_{j})\le D^{2}\phi(x_{j}).
\]
We also know that $\psi_{\varepsilon}^{-}\leq u_{\varepsilon}\leq\psi_{\varepsilon}^{+}$.
If $\psi_{\varepsilon}^{-}(x_{0})=u_{\varepsilon}(x_{0})$ then (\ref{eq: 2 in Reg u_e})
holds trivially. So suppose $\psi_{\varepsilon}^{-}(x_{0})<u_{\varepsilon}(x_{0})$.
Then for large $j$ we have $\psi_{\varepsilon}^{-}(x_{j})<\tilde{u}_{j}(x_{j})$.
Hence by ellipticity of $F$ and equation (\ref{eq: u-e,d}), at $x_{j}$
we have 
\begin{align*}
F(x_{j},\tilde{u}_{j},D\phi,D^{2}\phi) & \le F(x_{j},\tilde{u}_{j},D\tilde{u}_{j},D^{2}\tilde{u}_{j})\\
 & =\beta_{\delta_{j}}(\psi_{\varepsilon}^{-}-\tilde{u}_{j})-\beta_{\delta_{j}}(\tilde{u}_{j}-\psi_{\varepsilon}^{+})=-\beta_{\delta_{j}}(\tilde{u}_{j}-\psi_{\varepsilon}^{+})\le0.
\end{align*}
Thus by letting $j\to\infty$ we see that (\ref{eq: 2 in Reg u_e})
holds in this case too. Similarly, we can show that when $\psi$ is
a $C^{2}$ function and $u_{\varepsilon}-\psi$ has a local minimum
at $x_{0}$, we have 
\[
\max\{\min\{F(x_{0},u_{\varepsilon},D\psi,D^{2}\psi),u_{\varepsilon}-\psi_{\varepsilon}^{-}\},u_{\varepsilon}-\psi_{\varepsilon}^{+}\}\ge0.
\]
Therefore $u_{\varepsilon}$ is a viscosity solution of equation (\ref{eq: eq instd of dbl obstcl u_e}).
Hence, as we have shown in Part II of the proof of Theorem \ref{thm: Reg dbl obstcl},
$u_{\varepsilon}$ is also a strong solution of (\ref{eq: eq instd of dbl obstcl u_e});
so it satisfies the double obstacle problem (\ref{eq: dbl obstcl u_e})
as desired.
\end{proof}
\begin{thm}
\label{thm: Reg u}Suppose $F$ satisfies Assumptions \ref{assu: F},\ref{assu: F Lip}.
Also, suppose $\psi^{\pm}$ satisfy Assumption \ref{assu: =00005Cpsi +-}.
Then the double obstacle problem (\ref{eq: dbl obstcl 2}) has a solution
$u$, and we have 
\[
u\in W_{\mathrm{loc}}^{2,\infty}(U).
\]
Furthermore, if $\partial U$ is $C^{2,\alpha}$ for some $\alpha>0$,
and $\psi^{\pm}$ are $C^{2,\alpha}$ on a neighborhood of $\partial U$
in $\overline{U}$, we have 
\[
u\in W^{2,\infty}(U).
\]
\end{thm}
\begin{proof}
Let $u_{\varepsilon}$ be as in Lemma \ref{lem: Reg u_e}. Let us
first show that 
\begin{equation}
|F[u_{\varepsilon}]|\le C+\frac{C}{d-\varepsilon},\qquad\qquad\textrm{a.e. on }U_{\varepsilon},\label{eq: F=00005Bu_e=00005D bdd}
\end{equation}
where $d$ is the Euclidean distance to $\partial U$, and $C$ is
independent of $\varepsilon$. (Note that by (\ref{eq: inclusions})
we have $U_{\varepsilon}\subset\{d>\varepsilon\}$.) To see this,
note that in the open set $\{\psi_{\varepsilon}^{-}<u_{\varepsilon}<\psi_{\varepsilon}^{+}\}$
we have $F[u_{\varepsilon}]=0$; so the desired bound holds trivially.
Next consider the set $\{u_{\varepsilon}=\psi_{\varepsilon}^{+}\}$.
By (\ref{eq: dbl obstcl u_e}) we have $F[u_{\varepsilon}]\le0$ a.e.
on $\{u_{\varepsilon}=\psi_{\varepsilon}^{+}\}$. On the other hand,
since both $u_{\varepsilon},\psi_{\varepsilon}^{+}$ are twice weakly
differentiable, for a.e. $x\in\{u_{\varepsilon}=\psi_{\varepsilon}^{+}\}$
we have $Du_{\varepsilon}(x)=D\psi_{\varepsilon}^{+}(x)$ and $D^{2}u_{\varepsilon}(x)=D^{2}\psi_{\varepsilon}^{+}(x)$.
Hence by the ellipticity of $F$ and the bound (\ref{eq: bd D2 psi _e})
for $D^{2}\psi_{\varepsilon}^{+}$ we have 
\begin{align*}
F(x,u_{\varepsilon}(x),Du_{\varepsilon}(x),D^{2}u_{\varepsilon}(x)) & =F(x,\psi_{\varepsilon}^{+}(x),D\psi_{\varepsilon}^{+}(x),D^{2}\psi_{\varepsilon}^{+}(x))\\
 & \ge F(x,\psi_{\varepsilon}^{+}(x),D\psi_{\varepsilon}^{+}(x),\tfrac{C_{2}}{d(x)-\varepsilon}I)\\
 & \ge F(x,\psi_{\varepsilon}^{+}(x),D\psi_{\varepsilon}^{+}(x),0)-\tfrac{n\Lambda C_{2}}{d(x)-\varepsilon}\ge-C-\tfrac{C}{d(x)-\varepsilon}.
\end{align*}
Note that $D\psi_{\varepsilon}^{+}$ is uniformly bounded by (\ref{eq: bd D psi _e}).
We can similarly show that $F[u_{\varepsilon}]$ has the desired bound
on $\{u_{\varepsilon}=\psi_{\varepsilon}^{-}\}$.

Now, we choose a decreasing sequence $\varepsilon_{k}\to0$ such that
$\overline{U}_{\varepsilon_{k}}\subset U_{\varepsilon_{k+1}}$ (this
is possible by (\ref{eq: U_e subst U_e'})). For convenience we use
$U_{k},u_{k},\psi_{k}^{\pm}$ instead of $U_{\varepsilon_{k}},u_{\varepsilon_{k}},\psi_{\varepsilon_{k}}^{\pm}$.
Consider the sequence $u_{k}|_{U_{2}}$ for $k>2$. By (\ref{eq: F=00005Bu_e=00005D bdd}),
(\ref{eq: inclusions}) we have 
\[
\|F[u_{k}]\|_{L^{\infty}(U_{2})}\le C
\]
for some $C$ independent of $k$. Thus by interior $W^{2,p}$ estimates
for fully nonlinear elliptic equations (see Theorem 4.2 of \citep{winter2009w2},%
\begin{comment}
needs a covering argument ?
\end{comment}
{} and the proof of Lemma \ref{lem: Reg u_e}) we have 
\[
\|u_{k}\|_{W^{2,p}(U_{1})}\le\bar{C}\big(\|F[u_{k}]\|_{L^{p}(U_{2})}+\|u_{k}\|_{L^{\infty}(U_{2})}\big)
\]
for some constant $\bar{C}$ independent of $k$. Therefore $u_{k}$
is bounded in $W^{2,p}(U_{1})$. Consequently for every $\tilde{\alpha}<1$,
$\|u_{k}\|_{C^{1,\tilde{\alpha}}(\overline{U}_{1})}$ is bounded independently
of $k$, because $\partial U_{1}$ is $C^{2}$.

Therefore there is a subsequence of $u_{k}$'s, which we denote by
$u_{k_{1}}$, that weakly converges in $W^{2,p}(U_{1})$ to a function
$\tilde{u}_{1}$. In addition, we can assume that $u_{k_{1}},Du_{k_{1}}$
uniformly converge to $\tilde{u}_{1},D\tilde{u}_{1}$. Now we can
repeat this process with $u_{k_{1}}|_{U_{3}}$ and get a function
$\tilde{u}_{2}$ in $W^{2,p}(U_{2})$, which agrees with $\tilde{u}_{1}$
on $U_{1}$. Continuing this way with subsequences $u_{k_{l}}$ for
each positive integer $l$, we can finally construct a $C^{1}$ function
$u$ in $W_{\textrm{loc}}^{2,p}(U)$ (note that $U=\bigcup U_{k}$
by (\ref{eq: U=00003DU U_e})). It is obvious that $\psi^{-}\le u\le\psi^{+}$,
since $\psi_{k}^{-}\le u_{k}\le\psi_{k}^{+}$ for every $k$.

Let us show that $u$ satisfies the double obstacle problem (\ref{eq: dbl obstcl 2}).
It suffices to show that $u$ satisfies 
\[
\max\{\min\{F[u],u-\psi^{-}\},u-\psi^{+}\}=0.
\]
Similarly to the proof of Lemma \ref{lem: Reg u_e},%
\begin{comment}
or using stability ?
\end{comment}
{} we can show that $u$ is a viscosity solution of the above equation.
Then, as we have shown in Part II of the proof of Theorem \ref{thm: Reg dbl obstcl},
it follows that $u$ is also a strong solution of the above equation;
so it satisfies the double obstacle problem (\ref{eq: dbl obstcl 2})
as desired.

Next, similarly to Part III of the proof of Theorem \ref{thm: Reg dbl obstcl},
by utilizing the bounds (\ref{eq: bd D2 psi _e}) for $D^{2}\psi_{k}^{\pm}$
and (\ref{eq: F=00005Bu_e=00005D bdd}) for $F[u_{k}]$, we can show
that $D^{2}u_{k}$ is bounded on $\{u_{k}=\psi_{k}^{\pm}\}$ independently
of $k$. Then we can apply the result of \citep{Indrei-Minne} to
deduce that $D^{2}u_{k}$ is locally bounded independently of $k$,
and conclude that $u$ belongs to $W_{\mathrm{loc}}^{2,\infty}(U)$.

Finally, suppose that $\partial U$ is $C^{2,\alpha}$, and $\psi^{\pm}$
are $C^{2,\alpha}$ on $\overline{U}\cap\{d<3r\}$. Let $0\le\zeta\le1$
be a $C^{\infty}$ function which equals $1$ on $U\cap\{d<r\}$ and
equals $0$ on $U\cap\{d>2r\}$. Let $\eta_{\varepsilon}$ be the
standard mollifier, and set 
\[
\hat{\psi}_{\varepsilon}^{\pm}:=\zeta\psi^{\pm}+(1-\zeta)(\eta_{\varepsilon}*\psi^{\pm})
\]
for $\varepsilon$ small enough. Note that $\hat{\psi}_{\varepsilon}^{\pm}$
are $C^{2,\alpha}$ on $\overline{U}$, and agree on $\partial U$.
Also, $\hat{\psi}_{\varepsilon}^{\pm}$ uniformly converges to $\psi^{\pm}$
as $\varepsilon\to0$. It is obvious that $\hat{\psi}_{\varepsilon}^{-}=\psi^{-}<\psi^{+}=\hat{\psi}_{\varepsilon}^{+}$
on $U\cap\{d<r\}$. Now on $U\cap\{d\ge\frac{1}{2}r\}$ we have $\psi^{+}-\psi^{-}\ge c>0$.
Hence if $d(x)\ge r$ we get 
\begin{align*}
\eta_{\varepsilon}*\psi^{+}(x)-\eta_{\varepsilon}*\psi^{-}(x)=\int_{|y|\leq\varepsilon}\eta_{\varepsilon}(y) & [\psi^{+}(x-y)-\psi^{-}(x-y)]\,dy\\
 & \qquad\qquad\ge c\int_{|y|\leq\varepsilon}\eta_{\varepsilon}(y)\,dy=c.
\end{align*}
Therefore we have 
\[
\hat{\psi}_{\varepsilon}^{+}-\hat{\psi}_{\varepsilon}^{-}:=\zeta(\psi^{+}-\psi^{-})+(1-\zeta)(\psi_{\varepsilon}^{+}-\psi_{\varepsilon}^{-})\ge c(\zeta+1-\zeta)=c>0.
\]
Thus we have $\hat{\psi}_{\varepsilon}^{-}<\hat{\psi}_{\varepsilon}^{+}$
on $U$. In addition, note that around $\partial U$, $D^{2}\hat{\psi}_{\varepsilon}^{\pm}=D^{2}\psi^{\pm}$
are bounded. Thus similarly to Lemma \ref{lem: D, D2 psi_e} we can
show that for any unit vector $\xi$ and every $x\in U$ we have 
\begin{equation}
|D\hat{\psi}_{\varepsilon}^{\pm}|\le C,\qquad\qquad\pm D_{\xi\xi}^{2}\hat{\psi}_{\varepsilon}^{\pm}(x)\le C\label{eq: bd D2 psi_til}
\end{equation}
for some $C$ independent of $\varepsilon$.

Now we can repeat the construction of $u_{\epsilon}$ with $\hat{\psi}_{\varepsilon}^{\pm}$
instead of $\psi_{\varepsilon}^{\pm}$. Note that in this case we
have $U_{\varepsilon}=U$ for every $\varepsilon$. Also, if we use
the bound (\ref{eq: bd D2 psi_til}) instead of (\ref{eq: bd D2 psi _e})
in the first paragraph of the proof of this theorem, we can conclude
that 
\[
|F[u_{\varepsilon}]|\le\tilde{C},\qquad\qquad\textrm{a.e. on }U
\]
for some $\tilde{C}$ independent of $\varepsilon$. Hence we can
deduce that $\|u_{\varepsilon}\|_{W^{2,p}(U)}$ is uniformly bounded.
Thus a subsequence of $u_{\varepsilon}$ converges to a function $u$.
Then we can repeat Parts II-IV of the proof of Theorem \ref{thm: Reg dbl obstcl}
to show that $u$ satisfies the double obstacle problem (\ref{eq: dbl obstcl 2}),
and we have $u\in W^{2,\infty}(U)$ as desired.
\end{proof}
\begin{acknowledgement*}
This research was in part supported by Iran National Science Foundation
Grant No 97012372.
\end{acknowledgement*}

\bibliographystyle{plainnat}
\bibliography{/Volumes/A/Dropbox/Bibliography-Jan-2021}

\end{document}